\documentclass[letterpaper]{article}
\usepackage{uai2020}
\usepackage[margin=1in]{geometry}

\usepackage{times}

\usepackage{amsfonts}

\usepackage{amsmath} % for math operator s.t. argmin

\usepackage{graphicx} % graphics and figures

\makeatletter
\DeclareRobustCommand{\qed}{%
  \ifmmode % if math mode, assume display: omit penalty etc.
  \else \leavevmode\unskip\penalty9999 \hbox{}\nobreak\hfill
  \fi
  \quad\hbox{\qedsymbol}}
\newcommand{\openbox}{\leavevmode
  \hbox to.77778em{%
  \hfil\vrule
  \vbox to.675em{\hrule width.6em\vfil\hrule}%
  \vrule\hfil}}
\newcommand{\qedsymbol}{\openbox}
\newenvironment{proof}[1][\proofname]{\par
  \normalfont
  \topsep6\p@\@plus6\p@ \trivlist
  \item[\hskip\labelsep\itshape
    #1.]\ignorespaces
}{%
  \qed\endtrivlist
}
\newcommand{\proofname}{Proof}
\makeatother

\usepackage{enumitem} % for letter enumeration

\usepackage{bbm} % for indicator function

\newcommand\independent{\protect\mathpalette{\protect\independenT}{\perp}}
\def\independenT#1#2{\mathrel{\rlap{$#1#2$}\mkern2mu{#1#2}}} 

\newcommand{\ipm}{\text{IPM}}
\newcommand{\mmd}{\text{MMD}}
\newcommand{\disc}{\text{disc}}
\newcommand{\ipmFctClass}{\mathcal{F}}
\newcommand{\responseFct}{\tau}
\newcommand{\fct}{f}
\newcommand{\wass}{\mathcal{W}_1}
\newcommand{\rkhs}{\mathcal{H}}
\newcommand{\kernel}{\mathcal{K}}
\newcommand{\inner}[2]{\langle#1,#2\rangle_{\rkhs}}
\newcommand{\regClass}{\mathcal{G}}
\newcommand{\regFct}{g}
\newcommand{\regFctMin}{\hat{\regFct}_{\wVect}^{\nObs}}
\newcommand{\convRegFct}{\rho}
\newcommand{\ipmCst}{\gamma}
\newcommand{\boundedGees}{B}

\newcommand{\w}{w}
\newcommand{\wVect}{\mathbf{w}}
\newcommand{\reg}[1]{R_{#1}}
\newcommand{\regHyp}{\lambda}
\newcommand{\simplex}{\Lambda_{\nObs}}
\newcommand{\clip}{W}

\newcommand{\nObs}{n}
\newcommand{\treat}{T}
\newcommand{\tr}{t}
\newcommand{\cov}{X}
\newcommand{\x}{x}
\newcommand{\out}{Y}
\newcommand{\y}{y}

\newcommand{\tdim}{d_{\treat}}
\newcommand{\xdim}{d_{\cov}}
\newcommand{\joinVar}{Z}
\newcommand{\z}{z}
\newcommand{\joindim}{d}
\newcommand{\treatSample}{\mathbf{T}}

\newcommand{\tSupp}{\mathcal{T}}
\newcommand{\txSupp}{\mathcal{Z}}
\newcommand{\ySupp}{\mathcal{Y}}
\newcommand{\tDistr}{P_{\treat}}
\newcommand{\xDistr}{P_{\cov}}
\newcommand{\joinDistr}{P_{\treat,\cov}}
\newcommand{\yCondDistr}{P_{\out|\treat,\cov}}
\newcommand{\yDistr}{P_{\out}}

\newcommand{\source}{P}
\newcommand{\target}{Q}
\newcommand{\txTarget}{\target_{\treat, \cov}}
\newcommand{\yConstRange}{M}
\newcommand{\yVar}{\sigma^2}
\newcommand{\anyDensity}[1]{f_{#1}}

\newcommand{\dirac}[1]{\delta_{#1}}
\newcommand{\wDistr}{\hat{\source}^{\wVect}_{\treat,\cov}}
\newcommand{\empTar}{\hat{\target}_{\treat, \cov}}
\newcommand{\souTarDensity}{\frac{\partial \txTarget}{\partial \joinDistr}}
\newcommand{\souTarDenBound}{\clip_0}

\newcommand{\epo}{\mu}

\newcommand{\cbdm}{\text{CBDM}}

\newcommand{\causalRisk}{\mathcal{R}^{\epo}_{\treat}}
\newcommand{\condExp}{\responseFct}

\newtheorem{assumption}{Assumption}
\newtheorem{definition}{Definition}
\newtheorem{theorem}{Theorem}
\newtheorem{lemma}{Lemma}
\newtheorem{remark}{Remark}
\newtheorem{proposition}{Proposition}
\DeclareMathOperator*{\argmin}{arg\,min}

\newcommand{\radVar}{\xi}
\newcommand{\radVarVect}{\boldsymbol{\xi}}
\newcommand{\radComp}{R_{\nObs}(\regClass \circ \treatSample)}

\newcommand{\approxDelta}{\Delta}

\title{A Balancing Weight Framework for Estimating \\the Causal Effect of General Treatments}

\author{} % LEAVE BLANK FOR ORIGINAL SUBMISSION.
\author{ {\bf Guillaume Martinet} \\
ORFE, Princeton University \\
}

\begin{document}

\maketitle

\begin{abstract}

In observational studies, weighting methods that directly optimize the balance between treatment and covariates have received much attention lately; however these have mainly focused on binary treatments. Inspired by domain adaptation, we show that such methods can be actually reformulated as specific implementations of a discrepancy minimization problem aimed at tackling a shift of distribution from observational to interventional data. More precisely, we introduce a new framework, \emph{Covariate Balance via Discrepancy Minimization} (\cbdm), that provably encompasses most of the existing balancing weight methods and formally extends them to treatments of arbitrary types (e.g., continuous or multivariate). We establish theoretical guarantees for our framework that both offer generalizations of properties known when the treatment is binary, and give a better grasp on what hyperparameters to choose in non-binary settings. Based on such insights, we propose a particular implementation of \cbdm~for estimating dose-response curves and demonstrate through experiments its competitive performance relative to other existing approaches for continuous treatments.
\end{abstract}

\section{PRELIMINARIES}
% (1)"From epidemiology, healthcare or genomics to social sciences" -> "Examples of such fields include: epidemiology, genomics, and political science, where the problem of causal effect "
% (2)"From epidemiology, healthcare or genomics to social sciences" -> "From epidemiology to political science to economics,"

Estimating the causal effect of a treatment or, more broadly speaking, a cause on an outcome is a common problem in many fields of practical importance. From epidemiology to political science to economics, the problem of causal effect estimation arises when we are interested in quantifying the effect of some potential action or treatment administration. In many situations though, the use of controlled experiments is hampered by either technical or ethical considerations, and we have to resort to \emph{observational} data where the treatment assignment is no longer under our control and can be confounded. Causal inference is then done by leveraging additional information such as pre-treatment covariates that we believe encompass most of the confounding effects present in the data. Over the years different ways of adjusting for these covariates have been introduced in the literature \cite{imbens2004nonparametric}, including: matching \cite{rosenbaum1989optimal, stuart2010matching, wu2018matching}, subclassification \cite{hirano2004propensity, imai2004causal,imbens2000role,rosenbaum1983central}, regression \cite{hill2011bayesian,rubin1977assignment}, weighting \cite{hirano2003efficient, robins2000marginal, rosenbaum1987model} or some doubly robust combinations of these \cite{kennedy2017non,robins1994estimation}.

As a preprocessing step before further analysis, weighting has the advantage to allow the practitioner to focus solely on modeling the dose response curve, and potentially use the same weights for different outcomes. Usually, weights are computed by estimating and inverting the \emph{propensity score}, that is the distribution of treatment given the covariates. However, low overlap or even mild misspecification of the propensity score model can lead to highly variable weights and unstable treatment effect estimation \cite{kang2007demystifying}. To tackle such issues, weighting methods that optimize the balance between treatment and covariates either directly \cite{chan2016globally,hainmueller2012entropy, zubizarreta2015stable} or in parallel to the estimation of a propensity score \cite{imai2014covariate, mccaffrey2004propensity} have been introduced. A simple example of this approach for a binary treatment would be to find weights that equate some moments of the covariates within each treatment group to their population-wise values. The idea is that if the response functions, i.e.~the expected outcome given the covariates for different treatment values, can be linearly approximated by the corresponding monomials, such weights are \emph{enough} to correct for confounding effects. In comparison, propensity score based weighting is unnecessarily harder and thus more prone to errors. Extensions to handle richer classes of functions like RKHS have been made \cite{hazlett2018kernel,kallus2016generalized, kallus2017framework, wong2017kernel}, and overall the balancing weight approach has been well-studied theoretically in binary settings \cite{fan2016improving, hirshberg2019minimax, wang2017minimal,zhao2019covariate,zhao2017entropy}. Indeed, the balancing weight approach has mostly been restricted to binary treatments and only recently some attempts have been made to extend the idea to non-binary settings \cite{ fong2018covariate,kallus2019kernel, yiu2018covariate, zhu2015boosting}. However, these extensions seem to focus only on minimizing the covariance or association between treatment and covariates, and often lack proper theoretical justification.

% (1) "The main reason of this relationship is that..." -> "This relationship exists because both the..."
% (2) "The main reason of this relationship is that..." -> "This relaitionship comes(draws) from the fact that..."
\paragraph{Contributions.} We argue in the present paper that the balancing weight approach should be interpreted as a \emph{discrepancy minimization} problem from domain adaptation \cite{cortes2011domain, cortes2014domain, cortes2019adaptation, mansour2009domain} that aims at re-weighting the data so that it mimics an experimentation where each individual is given a random treatment value. This relationship draws from the fact that both the balancing weight approach and discrepancy minimization incorporate knowledge or some assumption about the response function into the derivation of the weights, contrary to inverse propensity score weighting which is essentially a density-ratio estimation \cite{arbour2019permutation}. We derive from this new interpretation a framework that, as we show, not only encompasses existing balancing weight methods but also lets us naturally and formally extend them to treatments of arbitrary types, and which benefits from similar theoretical guarantees to the binary treatment setting.  
%and this with similar theoretical guarantees than in settings with binary treatments.

Based on insights offered by our analysis, we propose a particular implementation of our framework for estimating dose-response curves that we test on experiments. This implementation is similar to the algorithm proposed in the recent paper \cite{kallus2019kernel}. Compared to this work though, they don't provide formal or theoretical justifications for their algorithm, and seem to focus only on minimizing a functional covariance based on polynomials between treatment and covariates, which as we shall discuss may not work in all circumstances.

% In the following paragraph introduce the support of the treatments 
\paragraph{Setting.} The observational data are composed of $\nObs$ i.i.d.~joint observations $(\treat_i, \cov_i, \out_i)_{i = 1}^{\nObs}$ of treatments $\treat$, pre-treatment covariates $\cov$ and outcomes $\out$. We work under the Neyman-Rubin potential outcomes framework \cite{neyman1923applications, rubin1974estimating, rubin2005causal}; that is, we posit the existence of potential outcomes $\out(\tr)$ that indicate what the outcome would have been if the treatment were fixed at $\tr$; the observed outcome is the potential outcome taken at the observed treatment: $\out = \out(\treat)$. Call $\tSupp$ the support of the distribution of $\treat$. The fact that the covariates capture all of the confounding effects is formalized by the ignorability assumption \cite{imai2004causal,rosenbaum1983central}:
% Ignorability assumes that blah
\begin{assumption} [Ignorability] \label{ass:ignorability}
Ignorability assumes:
\[
\out(\tr) \independent \treat | \cov, \quad \forall t \in \tSupp.
\]
\end{assumption}
The causal quantity of interest is the expected potential outcome, also called the dose-response function \cite{fong2018covariate}, which represents the outcome we would get on average if we fixed the treatment at $\tr$ for the considered population:
\begin{definition} [Dose-Response Function] \label{def:epo}
%Calling $\tSupp$ as in Assumption \ref{ass:ignorability}, the dose-response function is:
\[
\epo(\tr) = \mathbb{E}[\out(\tr)], \quad \forall \tr \in \tSupp.
\]
\end{definition}
In most cases, because of the dependence between treatment assignment $\treat$ and potential outcomes $\out(t)$, we will have $\mathbb{E}[\out | \treat = \tr] \neq \epo(\tr)$: This is the reason why causation is not the same as prediction. It turns out that $\epo(\tr)$ is \emph{also} actually a conditional expectation under a \emph{shifted} distribution. To see this, first define $P_V$ and $P_{W|V}$ as the distribution and conditional distribution respectively of random variables $V$ and $W$, and let:

\begin{itemize}
    \item $\source = \joinDistr \otimes \yCondDistr$, be the distribution from which our observational data is drawn.
    \item $\target = \tDistr \otimes \xDistr \otimes \yCondDistr$, be the distribution of our system in an experimental setting, where $\treat \independent \cov$.
\end{itemize}
If $\epo(\tr)$ is not equal to the conditional expectation of $\out$ given $\treat$ under $\source$, it is in fact equal under $\target$. Using Assumption \ref{ass:ignorability} and manipulations of conditional expectations:
\begin{equation}\label{eq:epoEquality}
    \epo(t) = \mathbb{E}_{\xDistr}[\mathbb{E}_{\yCondDistr}[\out|\tr,\cov]] = \mathbb{E}_{\target}[\out |\treat = \tr].
\end{equation}

\begin{figure}[]
\centering
\includegraphics[width=6.5cm, height = 2.35cm]{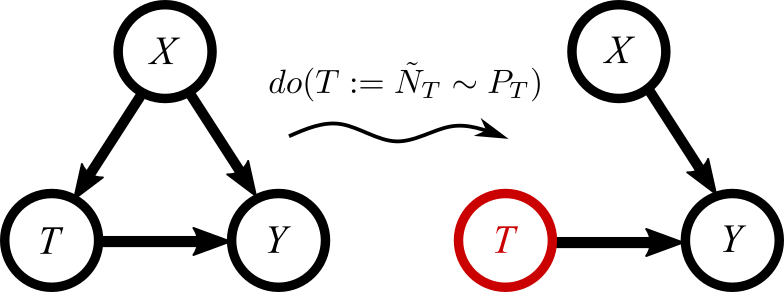}

\caption{A causal graph representation of the shift of distribution from $\source$ to $\target$.} \label{fig:causal_graph}
\end{figure}

\paragraph{Causal Graph Interpretation.} Another way of understanding equation \eqref{eq:epoEquality} is through the graphical interpretation of potential outcomes discussed in \cite[Section Section 3.6.3]{pearl2009causality} and \cite[Section 6.9.2]{peters2017elements}. Figure \ref{fig:causal_graph} (L.H.S.) displays a simple structural causal model representing our observational data distribution $\source$ (other models are also consistent with Assumption \ref{ass:ignorability}). Observe that $\epo(\tr)$ corresponds to the average outcome we obtain by setting $\treat$ to $\tr$ through a do-intervention that breaks the influence of $\cov$ on $\treat$. Equation \eqref{eq:epoEquality} means that we can estimate $\epo$ by regressing $\out$ on $\treat$ under the corresponding interventional distribution $\target$, represented in the R.H.S.~of Figure \ref{fig:causal_graph}, where each individual receives a random treatment like in an experiment. As we observe data only from $\source$ (and \emph{not} from $\target$), this is a domain adaption problem.
% (bring up Q in the last sentence, maybe some extra words to indicate the nature of it being domain adaptation). 

\paragraph{Additional Notations.} In the following, we call $\responseFct(\tr,\x) = \mathbb{E}_{\out | \treat, \cov}[\out|\tr,\x]$ the response function. We call $\joinVar = (\treat, \cov)$ and assume the supports $\txSupp \subset \mathbb{R}^{\tdim} \times \mathbb{R}^{\xdim}$ of $\joinDistr$ and $\ySupp \subset \mathbb{R}$ of $\yDistr$ to be compact. Hence there are $\yConstRange, \yVar \geq 0$ such that a.s., $Var_{\out|\treat,\cov}(\out |\tr,\x) \leq \yVar$, and $\forall \y \, \yCondDistr \text{-a.s.}$, $| y - \responseFct(\treat,\cov)| \leq \yConstRange / 2$, where 'a.s' means 'almost surely'. We also denote the simplex in $\mathbb{R}^\nObs$ by $\simplex = \{\wVect \in \mathbb{R}^{\nObs} : \wVect \geq 0, \,\, \sum_i^{\nObs} \w_{i} = 1\}$. Finally, $\|.\|_1$, $\|.\|_2$ and $\|.\|_{\infty}$ refer to the $L^1$, $L^2$ and $L^{\infty}$ norms.

\section{\cbdm~FRAMEWORK}

First, let's recall that the goal is to estimate the dose-response function $\epo$. As it is a conditional expectation under $\target$, we will measure the (relative) success of any function $\regFct \in \regClass$ in approximating $\epo$ by using the quadratic risk defined as follows:
\begin{equation} \label{eq:causalRisk}
    \forall \regFct \in \regClass, \quad \causalRisk(\regFct) = \mathbb{E}_{\target_\treat}[(\epo(\treat) - \regFct(\treat))^2].
\end{equation}

As mentioned previously we have $\mathbb{E}_{\source}[\out | \treat = \tr] \neq \epo(\tr)$, therefore we cannot estimate $\epo$ simply by returning $\regFct$ that minimizes the empirical risk $\sum_i \nObs^{-1} (\out_i - \regFct(\treat_i))^2$. Equation \eqref{eq:epoEquality} though tells us we would be able to do so by using a sample drawn from $\target$, but we have only access to observational data. We can however re-weight our data so that the resulting weighted empirical risk would be close enough to the risk under $\target$. Therefore, a weighting procedure for the estimation of the dose-response function proceeds in two steps:

- { {\bf Step 1:}}  Obtain weights $\wVect$ that handles the shift of distribution from $\joinDistr$ to $\txTarget$ (e.g.~the \cbdm~weights from Definitions \ref{def:CBDMframe} or \ref{def:cbdmIPM} afterward). 

- { {\bf Step 2:}} Return $\regFct \in \regClass$ that minimizes the weighted empirical risk $\sum_i \w_i (\out_i - \regFct(\treat_i))^2$.

Notice the weights from Step 1 need only to depend on $(\treat, \cov)$ because $\source_{\out | \treat, \cov} = \target_{\out | \treat, \cov}$.

An obvious choice of weights for Step 1 would be to take the density ratio $\nObs\w_i = \partial \txTarget / \partial \joinDistr (\treat_i, \cov_i)$ (when it exists), and several methods from the domain adaptation literature are available for estimating it \cite{bickel2007discriminative, huang2007correcting, sugiyama2008direct}. This approach explored in \cite{arbour2019permutation} is actually another way of deriving inverse propensity weights: Let $\anyDensity{}$ be any density, we have $\partial \txTarget / \partial \joinDistr = \anyDensity{\treat} \anyDensity{\cov} / \anyDensity{\treat, \cov} = \anyDensity{\treat} / \anyDensity{\treat | \cov}$ which is indeed the (stabilized) inverse propensity score \cite{robins2000marginal}. The intrinsic purpose of balancing weights, on the other hand, is to control for the empirical average of a specific class of functions. Because these weights solve a weaker problem than the inverse propensity score, when such a class is well-specified w.r.t.~the response function $\responseFct$ they tend to be more stable and efficient. It turns out that the counterpart of such an approach in the domain adaptation literature is the so-called \emph{discrepancy minimization} algorithm first introduced by \cite{mansour2009domain}. The algorithm proceeds by deriving weights that minimize a discrepancy notion between source and target data that incorporates the hypothesis class $\regClass$ and the loss chosen for the regression. Many variants of this kind of notion have been introduced in the literature, such as: the $d_{\mathcal{A}}$-distance \cite{ben2007analysis}, the discrepancy  \cite{cortes2011domain, cortes2014domain,mansour2009domain}, the $\mathcal{Y}$-discrepancy \cite{mohri2012new}, the generalized discrepancy \cite{cortes2019adaptation}, and integral probability metrics (\ipm) \cite{zhang2012generalization}. We call \emph{Covariate Balance via Discrepancy Minimization}, or \cbdm, the framework containing all weighting methods that are implementations of the discrepancy minimization algorithm applied to treatment effect estimation. Calling the weighted empirical distribution $\wDistr = \sum_i \w_i \dirac{\treat_i, \cov_i}$, where $\dirac{}$ refers to the Dirac measure, and $\empTar$ an empirical estimation of $\txTarget$, \cbdm~is defined as follows:
\begin{definition}[CBDM Framework] \label{def:CBDMframe} A method falls into the \cbdm~framework if it returns weights $\wVect$ s.t.:
 \begin{equation} \label{eq:generaCBDM}
     \wVect = \argmin_{\wVect' \in \mathbb{R}^\nObs} \quad \disc(\hat{\source}^{\wVect'}_{\treat,\cov}, \empTar) + \reg{\regHyp}(\wVect'),
 \end{equation}
where $\disc$ refers to a discrepancy notion between measures, and $\reg{\regHyp}$ is a regularizer parameterized by $\regHyp \geq 0$. 
\end{definition}
In the following we focus solely on integral probability metrics, for two reasons. First, these discrepancy notions are easily implementable and interpretable. Secondly, they have been used in binary treatment effect estimation \cite{ kallus2016generalized,kallus2017framework}, yet without any reference to the discrepancy minimization approach. We also discuss how we can create $\empTar$ from the available observational data at the end of this section.
 
\paragraph{Integral Probability Metrics.} 

\ipm s quantify how much two probability measures differ by their expectations over some class $\ipmFctClass$ of measurable functions \cite{muller1997integral}:
\begin{definition}[\ipm] \label{def:IPM}
Let $\source_0$ and $\target_0$ be two probability measures. Their \ipm~value over some class $\ipmFctClass$ is:
\[
\ipm_{\ipmFctClass}(\source_0, \target_0) = \sup_{\fct \in \ipmFctClass} \left | \mathbb{E}_{\joinVar \sim \source_0}[\fct(\joinVar)] - \mathbb{E}_{\joinVar \sim \target_0}[\fct(\joinVar)] \right |.
\]
\end{definition}

We consider two types of \ipm s that are often used in the literature for comparing two distributions: The maximum mean discrepancy (\mmd) \cite{gretton2007kernel, gretton2012kernel} and the Wasserstein-1 distance ($\wass$) \cite{ramdas2017wasserstein, villani2003topics}. First, \mmd~is $\ipm_{\ipmFctClass}$ with $\ipmFctClass$ being the unit ball of a RKHS $\rkhs$. Recall that a RKHS is a Hilbert functional space fully characterized by a kernel $\kernel(\cdot,\cdot)$ in the sense that \cite{aronszajn1950theory}: $\forall \z \in \txSupp$, $\kernel(\z,\cdot)$ is a feature embedding of $\z$ into $\rkhs$; $\forall \fct \in \rkhs, \fct(\z) = \inner{\kernel(\z,\cdot)}{\fct}$; and $\rkhs$ is the completion of $\text{span}\{ \kernel(\z,\cdot): \z \in \txSupp\}$. We consider the following well-known kernels:

- {\bf Polynomial Kernels:} $\kernel(\z,\z') = (1 + \z^T \z')^p$ for some $p$. Its RKHS spans the space of all polynomials of degree up to $p$.

- {\bf Exponential Kernel:} $\kernel(\z,\z') = \exp(\z^T \z')$.

- {\bf Gaussian Kernel:} $\kernel(\z,\z') = \exp(-\| \z - \z' \|^2/2)$.

Second, the distance $\wass$ is also $\ipm_{\ipmFctClass}$, but instead with $\ipmFctClass$ being the class of the Lipschitz functions (for the Euclidian distance here) with Lipschitz constant equal at most one. This result is also known as the Kantorovich-Rubinstein theorem \cite{villani2003topics}. Even though \ipm s seem to focus only on some restricted class of functions $\ipmFctClass$, by linearity of the expectation they actually control for more than that. We formalize this fact, which is going to be useful in our analysis, in the following definitions:
\begin{definition}[Approximable Functions] \label{def:approxFct} Let $\fct$ be a function on $\txSupp$, and $\ipmFctClass$ a class of functions, we say that:

- $\fct$ is ($\gamma$, $\epsilon$)-approximable by $\ipmFctClass$ if, for some $k$, there exist $\{ \fct_{i} \}_{i=1}^{k} \in \ipmFctClass^k$ and $\boldsymbol{\lambda} = \{ \lambda_{i} \}_{i=1}^{k} \in \mathbb{R}^k$ such that:
\[
 \| \boldsymbol{\lambda} \|_1 \leq \gamma \quad \text{and} \quad  \| \fct - {\textstyle\sum}_{i=1}^{k} \lambda_i \fct_i  \|_{\infty} \leq \epsilon. 
\]
- $\fct$ is approximable by $\ipmFctClass$ if for any $\epsilon > 0$ it is ($\gamma$, $\epsilon$)-approximable for some $\gamma \geq 0$.
\end{definition}

Recall that a continuous kernel $\kernel(\cdot,\cdot)$ is called \emph{universal} if the space of continuous functions on $\txSupp$ (compact), that is $C(\txSupp)$, is approximable by its RKHS \cite{steinwart2001influence}. For instance, the exponential and gaussian kernels are universal, but the polynomial kernel is not.

\paragraph{\ipm-Based \cbdm~Weights.} We develop \cbdm~with \ipm s and leave the study of other choices of discrepancy notions as an open question for future research. Following standard practice in causal inference, we also use the following constraints and regularizer:
\begin{definition}[\cbdm~with \ipm s] \label{def:cbdmIPM} Let $\regHyp \geq 0$ and some clipping constant $\clip \geq 0$; then define:
\[
\wVect = \argmin_{\wVect' \in \simplex; \,\, \w_i \leq \clip / \nObs, \, \forall i} \ipm^2_{\ipmFctClass}(\hat{\source}^{\wVect'}_{\treat,\cov}, \empTar) + \regHyp \reg{\nObs}(\wVect'),
\]
where $\reg{\nObs}(\wVect) = \nObs^{-2} \sum_{i=1}^{\nObs} \rho(\nObs \w_i)$, with $\rho$ strictly convex, continuous on $[0 , \clip]$ and differentiable on $(0,\clip)$. 
\end{definition}
Notice: $\reg{\nObs}(\wVect) = \| \wVect \|^2_2$ if $\rho(x) = x^2$ and $\reg{\nObs}(\wVect) =  \nObs^{-1} (\sum_{i=1}^{\nObs} \w_i \log(\w_i) + \log(\nObs))$ for $\rho(x) = x \log(x)$.

\paragraph{About $\empTar$.} Contrary to the usual setting in domain adaptation, we do not have access to samples from the target distribution $\txTarget$. Following \cite{arbour2019permutation}, it is however possible to build an empirical approximation of $\txTarget$, called $\empTar$, using only our observed data $(\treat_i, \cov_i)_{i = 1}^n$. For instance, we can just take the product of the marginal empirical distributions of $\treat$ and $\cov$: $\empTar = \hat{\source}_{\treat} \otimes \hat{\source}_{\cov}$. For continuous treatments this can be composed of $O(\nObs^2)$ Diracs, which may be computationally prohibitive for large $\nObs$. Instead, we found in practice that it is sufficient to simply shuffle the treatment values. More precisely, we create $\empTar$ by shuffling the treatments $\treat_i$ in the data while keeping the covariates $\cov_i$ fixed, we repeat the process $K$ times and merge the obtained data-sets, $\empTar$ is the resulting empirical distribution.

\section{SPECIAL CASES} \label{sec:special_cases}

In this section, to sustain the idea that the balancing weight approach should indeed be restated as a discrepancy minimization problem, we show that actually most of the weighting methods that directly optimize the balance between $\treat$ and $\cov$ are special cases of the \cbdm~framework. All proofs are in the appendix.

\subsection{Approximate and Exact Balancing} \label{sec:approxBal}
When the treatment is binary ($\treat \in \{ 0,1\}$), approximate and exact balancing methods seek to equate expectations of a finite number of functions $(f_k)_{k=1}^K$ within each treatment group to their population-wide counterparts. Formally, assume w.l.o.g.~that the $\nObs_0>0$ first observations are in the control group (where $\treat = 0$), and the last $\nObs_1 = \nObs - \nObs_0>0$ are the treated group ($T=1$). These methods commonly find $(\Tilde{\wVect}^0, \Tilde{\wVect}^1)$ such that:
\begin{equation} \label{eq:approxExactBal}
     \Tilde{\wVect}^\tr(\delta_{\tr}) = \argmin_{\substack{\nObs \Tilde{\wVect} / \nObs_\tr \in \Lambda_{\nObs_\tr}: \,\, \forall k, \\ |\nObs_\tr^{-1} \sum_{i=1}^{\nObs_t} \nObs \Tilde{\w}_i f_k(X_{i + \tr. \nObs_{0}}) - \Bar{f}_k | \leq \delta_{\tr}} } \|  \Tilde{\wVect} \|_2^2,
\end{equation}
for $\tr \in \{ 0,1\}$ and $\bar{f}_k = \nObs^{-1} \sum_{i=1}^{\nObs} f_k(\cov_i)$. Then, we estimate the dose-response function by computing the weighted average of $\out$ within each group, namely:
\[
\forall \tr \in \{0,1 \}, \quad \Tilde{\epo}(\tr) = \sum_{i = 1}^{\nObs_\tr} \Tilde{\w}^\tr_{i}(\delta_{\tr}) \out_{i + \tr . \nObs_{0}} \Big/  \sum_{i = 1}^{\nObs_\tr} \Tilde{\w}^\tr_{i}(\delta_{\tr}).
\]
If we set $\regClass$ to be the set of all functions defined on $\{ 0,1 \}$, $\Tilde{\epo}$ is indeed the minimizer from Step 2. When $\delta_\tr > 0$, this approach is called approximate balancing \cite{wang2017minimal,zubizarreta2015stable} and when $\delta_\tr = 0$, this is exact balancing \cite{chan2016globally}. Replacing the objective of \eqref{eq:approxExactBal} with the negative entropy yields the so-called entropy balancing method \cite{hainmueller2012entropy}. Overall, these methods are versions of \cbdm~:
\begin{proposition} \label{prop1}
Let $f_k^0(t,x) = (1 - t) f_k(x)$ and $f_k^1(t,x) = t f_k(x)$. Consider $\ipmFctClass = \{ f_k^t: t \in \{0,1\},k \in  \{ 1, \ldots, K \} \}$ and the \cbdm~weights for $\regHyp>0$:
\[
\wVect(\regHyp) = \argmin_{\wVect \in \Lambda_{0,1}} \ipm^2_{\ipmFctClass}(\wDistr, \empTar) + \regHyp \| \wVect \|_2^2,
\]
where $\Lambda_{0,1} = \{ \wVect = (\wVect^0, \wVect^1) \geq 0: \sum_{i=1}^{n_t} w_{i}^t = n_t/n, \,\, t\in \{0,1\} \}$ and $\empTar = \hat{\source}_{\treat} \otimes \hat{\source}_{\cov}$. Then for any $\regHyp > 0$, there is $(\delta_0, \delta_1) > 0$ such that:
\[
\forall t\in \{0,1 \}, \quad \wVect^t(\regHyp) = \Tilde{\wVect}^\tr( \delta_\tr).
\]
\end{proposition}

\subsection{Generalized Optimal Matching (GOM)}\label{sec:GOM}
The GOM framework~\cite{kallus2016generalized, kallus2017framework} estimates expected potential outcomes for the treated group, that is $\mathbb{E}[\out(t)|\treat = 1]$. For $\tr = 1$, this is simply $\mathbb{E}[Y|T=1]$, so weights are needed only for the control group. GOM extends the methods described in Section \ref{sec:approxBal} to larger functional classes and a common version is:
\[
\Tilde{\wVect}^{0}(\regHyp) = \argmin_{\Tilde{\wVect} \in \Lambda_{\nObs_0}} \ipm^2_{\ipmFctClass_{0}}(\hat{\source}_{\cov|\treat = 0}^{\Tilde{\wVect}},\hat{\source}_{\cov|\treat = 1}) + \regHyp \| \Tilde{\wVect} \|_2^2,
\]
where we consider $\hat{\source}_{\cov|\treat = 0}^{\Tilde{\wVect}} = \sum_{i=1}^{\nObs_0} \Tilde{w}_i \dirac{\cov_i}$, $\hat{\source}_{\cov|\treat = 1} = n_1^{-1} \sum_{i = n_0 +1}^{n} \dirac{\cov_i}$ and $\ipmFctClass_0$ is the unit ball of some normed functional space -- if it is an RKHS, then typically $\ipm_{\ipmFctClass_0} = \mmd$. We select target distribution $\target$ such that $\target_\cov = \source_{\cov | \treat = 1}$. With this minor change, GOM is therefore also a special case of \cbdm~:
\begin{proposition} \label{prop2}
Define $f^0(t,x) = (1 - t) f(x)$ and $f^1(t,x) = t f(x)$ for any $f \in \ipmFctClass_0$. Consider the \cbdm~weights $\wVect(\regHyp)$ defined in Prop.~\ref{prop1} but instead with $\ipmFctClass = \{ f^t: f \in \ipmFctClass_0, \,\, t \in \{0,1\}\}$ and $\empTar = \hat{\source}_{\treat} \otimes \hat{\source}_{\cov|\treat = 1}$. Then we have:
\[
\wVect^0(\regHyp) = \frac{n_0}{n} \Tilde{\wVect}^0(\regHyp), \quad \text{and} \quad 
\wVect^1(\regHyp) = \nObs^{-1} \mathbf{1},
\]
and the two methods return the same estimator $\Tilde{\epo}(\tr)$.
\end{proposition}

The GOM framework itself encompasses many existing matching procedures such as: nearest neighbor matching, caliper matching, coarsened exact matching and mean-matched sampling (see \cite{kallus2016generalized,kallus2017framework}). By transitivity, all such matching methods can therefore be expressed as special cases of the \cbdm~framework.

\subsection{Non-parametric CBGPS} \label{sec:npCBPS}
\cite{fong2018covariate} introduced a non-parametric extension of CBPS from \cite{imai2014covariate} to handle continuous (univariate) treatments. It seeks weights that maximizes the empirical log-likelihood while setting the expectations of $\treat$ and $\cov$, and their covariance to zero. \cite{yiu2018covariate} consider minimizing $\| \wVect \|^2_2$ in place of the empirical log-likelihood. A slightly weaker version of npCBGPS can hence be written this way, as \cbdm~weights:
\[
\wVect = \argmin_{\wVect \in \simplex}  \ipm^2_{\ipmFctClass}(\wDistr, \empTar) + \regHyp \| \wVect \|_2^2,
\]
where $\empTar = \hat{\source}_{\treat} \otimes \hat{\source}_{\cov}$ and $\ipmFctClass = \{t, \, x^{(k)}, \, t.x^{(k)}: k \in [1,\xdim]\}$, and $\cov$ has dimension $\xdim$ ($\treat$ is univariate here). 

\section{THEORY}

In this section, we provide further evidence of the fact that the \cbdm ~framework offers the right generalization of the balancing weight approach as we establish theoretical guarantees for \cbdm~that can be deemed as extensions of properties known for binary treatments.

\subsection{Causal Learning Bound} \label{sec:causalBounds}

As a justification for GOM (see Section \ref{sec:GOM}), \cite{kallus2016generalized, kallus2017framework} derived a bound on the difference between the weighted outcome average in the control group and $\mathbb{E}[\out(0)|\treat = 1]$ using an integral probability metric $\ipm_\ipmFctClass$ such that $\responseFct(0,\cdot) \in \ipmFctClass$. However, the difficulty in the current setting with arbitrary treatment type is that we don't have access to an explicit estimator of $\epo$ (such like the weighted average for GOM). Instead, we have to include a regression step in the analysis (Step 2). Using classical tools from learning theory, we derive a bound for \cbdm~(Def. \ref{def:cbdmIPM}) that provides, as with the bound for GOM, some insights on what class of functions $\ipmFctClass$ to choose. First, let's recall the definition of the Rademacher complexity.
\begin{definition} [Rademacher Complexity] \label{def:rademacher}
For some fixed array $\treatSample = \{ \treat_1, \ldots, \treat_\nObs \}$ of treatment values and $\radVarVect = (\radVar_i)_{i = 1}^{\nObs}$ i.i.d.~Rademacher variables (i.e.,~$\mathbb{P}(\radVar_i = -1) = \mathbb{P}(\radVar_i = +1) = 1/2$). Let:
\[
\radComp = \frac{2}{\nObs} \mathbb{E}_{\radVarVect}\left[\sup_{\regFct \in \regClass} \sum_{i =1}^{\nObs} \radVar_i \regFct(\treat_i) \right]. 
\]
\end{definition}
The Rademacher complexity is a classical notion in learning theory \cite{shalev2014understanding}. For a wide range of predictor classes, from linear regression to kernel methods to neural networks \cite{bartlett2002rademacher}, the Rademacher complexity converges to zero (usually at speed $O(\sqrt{(\log(\nObs)/ \nObs)}$). 
Next, call $\regFctMin \in \regClass$ the minimizer of the weighted empirical risk $\sum_i \w_i (\out_i - \regFct(\treat_i))^2$ from Step 2 and from now on, we will say that $\regClass$ is $\boundedGees$-bounded if $\forall \regFct,\regFct' \in \regClass$,  $|\regFct(\z) - \regFct'(\z)| \leq \boundedGees$ on $\txSupp$. Under Assumption \ref{ass:ignorability}:
\begin{theorem} \label{thm:causalBound} Assume $\forall \regFct \in \regClass, -2\condExp\regFct + \regFct^2$ is ($\gamma/2$, $\epsilon/2$)-approximable by $\ipmFctClass$ and that $\regClass$ is $\boundedGees$-bounded. Consider fixed data $(\treat_i,\cov_i)_{i = 1}^{\nObs}$ and $\wVect$ some weights built on them such that $\w_i \leq \clip/\nObs, \forall i$. We have $\forall \nObs > 0, \, \delta \in (0,1]$, with probability at least $1-\delta$ over the outcomes $\out_i$ :
\begin{align*}
   \causalRisk(\regFctMin) \leq & \inf_{\regFct \in \regClass} \causalRisk(\regFct)+ \clip\yConstRange \cdot \left(  \radComp + \frac{\boundedGees^2_\delta}{12 \boundedGees \nObs} \right)\\
   & + \ipmCst . \ipm_{\ipmFctClass}(\wDistr,\txTarget) + \sigma \boundedGees_{\delta} \| \wVect \|_2  + \epsilon ,
\end{align*}
where $\boundedGees_{\delta} = 4 \boundedGees \sqrt{2 \log(1/\delta)}$.
\end{theorem}
Any $\ipm$ respects the triangle inequality, in particular $\ipm(\wDistr,\txTarget) \leq \ipm(\wDistr,\empTar) + \ipm(\empTar,\txTarget)$. The term $\ipm(\empTar,\txTarget)$ does not depend on the weights and will usually converge to zero at speed $O(\nObs^{-1/2})$ when $\ipm = \mmd$ \cite{gretton2012kernel}, and at speed $O(\nObs^{-1/ \joindim})$ when $\ipm = \wass$ \cite{weed2019sharp}. Therefore, the \cbdm~weights (Def. \ref{def:cbdmIPM}), with $\reg{\nObs}(\wVect) = \| \wVect \|^2_2$, aim at minimizing the above upper bound. Also, the bound in Theorem \ref{thm:causalBound} gives some insight on how to choose $\ipmFctClass$, and why it might not be enough minimizing some functional covariance between $\treat$ and $\cov$. Indeed, this would correspond to controlling for a class $\ipmFctClass$ composed of products of functions that depend only on either $\tr$ or $\x$ (e.g.~see Section \ref{sec:npCBPS}): $\responseFct(\tr,\x)$ might not always be well approximated by such a class of functions.

\subsection{A Dual Approach} \label{sec:firstRobust}

It has been shown recently that approximate and exact balancing methods (Section \ref{sec:approxBal}) admit a dual formulation: these weights actually correspond to an implicit propensity score model fitted with some specific loss function \cite{wang2017minimal,zhao2019covariate,zhao2017entropy}. Interestingly, this implies that if this model is well specified, regardless of the response function $\responseFct$ the treatment effect estimator will converge \cite{zhao2017entropy}, and if the outcome model (i.e.~$\ipmFctClass$) is also well-specified the estimator becomes semi-parametric efficient \cite{wang2017minimal,zhao2017entropy}. We extend this dual interpretation to non-binary treatments as we show that the \cbdm~weights from Def.~\ref{def:cbdmIPM} admit a similar dual formulation. Call $\rho^*(x) = \sup_{y} yx - \rho(y)$ the Legendre transform of $\rho$.

\begin{theorem} \label{thm:dualProblem}
Let $\ipm_\ipmFctClass = \mmd_\kernel$ where $\kernel$ is a continuous kernel and $\rho(x) = +\infty$ outside $[0,\clip]$. A dual problem of the \cbdm~weights $\wVect$ from Definition \ref{def:cbdmIPM} is:
\begin{align} \label{eq:dualProblem}
    \min_{\mu, \boldsymbol{\alpha}} \,\,\, \frac{1}{\nObs} &\sum_i \rho^*\left(- \mu - \boldsymbol{\alpha}^{T} \mathbf{k}(\joinVar_i) \right) + \mu \\
    & + \frac{\regHyp}{4 \nObs} \boldsymbol{\alpha}^T \mathbf{K} \boldsymbol{\alpha} + \boldsymbol{\alpha}^T   \mathbb{E}_{\empTar}[\mathbf{k}(\joinVar)], \nonumber
\end{align}
where $\boldsymbol{\alpha} = (\alpha(z))_{z \in \hat{\txSupp}}$, $\mathbf{k}(\cdot) = (\kernel(z,\cdot))_{z \in \hat{\txSupp}}$, $\mathbf{K} = (\kernel(z,z'))_{z,z' \in \hat{\txSupp}}$ and $\hat{\txSupp} = \text{supp}(\hat{\source}^\nObs_{\joinVar}) \bigcup \text{supp}(\empTar)$. If $\mu^0$ and $\boldsymbol{\alpha}^0$ are solutions of \eqref{eq:dualProblem} then we have $\forall i$, $n \w_i = \frac{d\rho^*}{d x}(- \mu^0 -   (\boldsymbol{\alpha}^0)^{T} \mathbf{k}(\joinVar_i))$. 

\end{theorem} 

%The above theorem is a consequence of both the representer theorem \cite{hofmann2008kernel} and of a more general result we proved in Appendix \ref{app:dualPb}, based on the Fenchel-Rockafellar Theorem.
Theorem \ref{thm:dualProblem} means that actually the \cbdm~weights (Def.~\ref{def:cbdmIPM}) comes from the estimation of an implicit model of the density ratio $\partial \txTarget / \partial \joinDistr$, fitted with the loss in \eqref{eq:dualProblem}. This model is characterized by a family of densities $\{ \frac{d\rho^*}{d x}(\mu_\fct + \fct): \fct \in \rkhs\}$, where $\mu_\fct \in \mathbb{R}$ is a normalizing constant and $\rkhs$ is the RKHS of $\kernel$. Note if $\rho(x) = x^2$ on $[0,\clip]$, we have $\frac{d\rho^*}{d x}(x) = (x)_{+} \land \clip / 2 $ and if $\rho(x) = x \log(x)$ on $[0,\clip]$ then $\frac{d\rho^*}{d x}(x) = \exp(x - 1)\land \clip $, i.e.~the model is a capped exponential RKHS model \cite{hofmann2008kernel}.

For simplicity, we will focus on parametric models, namely the RKHS $\rkhs$ of $\kernel$ is of finite dimension.

\begin{definition}[Well-Specified Model] \label{def:wellSpec}
Let $\ipm_\ipmFctClass = \mmd_\kernel$, where $\kernel$ is continuous with finite-dimensional RKHS $\rkhs$. We say that the implicit model of the \cbdm~weights in Def.~\ref{def:cbdmIPM}, is well-specified if $\partial \txTarget / \partial \joinDistr = \frac{d\rho^*}{d x}(\mu + \fct)$ with $\mu \in \mathbb{R}$ and $\fct \in \rkhs$.
\end{definition}

\subsection{Consistency Results}

Combining our previous results, we derive sufficient conditions for our dose-response curve estimator $\regFctMin$ to minimize the (causal) risk $\causalRisk$ asymptotically. We assume:
\begin{assumption}[Overlap]\label{ass:overlap} Let $\joinDistr$ be the distribution of $(\treat,\cov)$, and $\txTarget = \tDistr \otimes \xDistr$. We assume:
\begin{enumerate} [label = (\alph*)]
    \item \label{ass:weakOverlap} Absolute continuity: $\txTarget \ll \joinDistr$. 
    \item \label{ass:strongOverlap} Bounded density: $\souTarDensity \leq \souTarDenBound$, for some $\souTarDenBound > 0$.
\end{enumerate}
\end{assumption}
Note, these assumptions are measure-theoretic formulations of the overlap assumptions commonly made in the literature \cite{athey2018approximate,chan2016globally,kallus2016generalized,  wong2017kernel}: When $\treat$ is binary, \ref{ass:weakOverlap} is equivalent to having $0<P(\treat = 1|\cov)<1$ and \ref{ass:strongOverlap} is equivalent to having $\eta<P(\treat = 1|\cov)<1-\eta$ for some $\eta>0$. Under Assumptions \ref{ass:ignorability} and \ref{ass:overlap}:

\begin{theorem} \label{thm:doubleConsistency} We focus on $\ipm_\ipmFctClass \in \{\wass, \mmd_\kernel\}$ where $\kernel$ is a continuous kernel.
Let $\wVect$ be the \cbdm~weights from Def. \ref{def:cbdmIPM} with $\clip > \souTarDenBound$ and $\empTar = \hat{\source}_{\treat} \otimes \hat{\source}_{\cov}$. Assume that $\regClass$ is $B$-bounded and that $\radComp$ (a.s.) converges to zero. Then, we have:
\[
\causalRisk(\regFctMin) \xrightarrow[n \rightarrow +\infty]{a.s.}  \inf_{\regFct \in \regClass} \causalRisk(\regFct), 
\]
if at least one of the following conditions is satisfied:
\begin{enumerate} [label = (\alph*)]
    \item  There exist two classes $\ipmFctClass_0$ and $\regClass_0$ such that: $\responseFct$~is approximable by $\ipmFctClass_0$, and all $\regFct \in \regClass$ are ($\gamma,0$)-approximable by $\regClass_0$ for some fixed $\gamma$ and $\{2 f.g + g'.g'': f \in \ipmFctClass_0, \text{ and } g,g',g'' \in \regClass_0 \cup \{ 0 \} \} \subset \ipmFctClass$.
    \item The implicit model of the \cbdm~weights is well-specified, as defined in Definition \ref{def:wellSpec}. Furthermore, assume that $\rho^*$ is strongly convex and smooth on any finite sub-interval of $(\frac{d\rho}{d x}(0^+),\frac{d\rho}{d x}(\clip^-))$. Note this assumption is true when $\rho(x) = x^2$ or $\rho(x) = x \log x$ on $[0,\clip]$.
\end{enumerate}
\end{theorem}

Finally, we show that for some $\ipm$s with flexible classes of functions $\ipmFctClass$, $\wDistr$ will indeed mimic asymptotically the interventional distribution $\target$, regardless of the form of $\responseFct$ or $\partial \txTarget / \partial \joinDistr $. Under Assumption \ref{ass:overlap}:

\begin{theorem} \label{thm:distrConv}
Consider $\wVect(\nObs)$ the \cbdm~weights (Def. \ref{def:cbdmIPM}) built on $(\joinVar_{i})_{i = 1}^{\nObs}$ with $\ipm \in \{ \wass, \mmd_{\kernel}\}$, where $\kernel$ is universal, $\clip > \souTarDenBound$ and $\empTar = \hat{\source}_{\treat} \otimes \hat{\source}_{\cov}$. We have :
\[
a.s. \quad \hat{\source}^{\wVect(\nObs)}_{} = \sum_{i = 1}^{\nObs} \w_{i}(\nObs) \dirac{(\joinVar_i,\out_i)} \xrightarrow[n \rightarrow +\infty]{d} \target .
\]
\end{theorem}

\section{IMPLEMENTATION DETAILS} \label{sec:implementation}

We now discuss implementation details for the \cbdm~weighting algorithm from Definition \ref{def:cbdmIPM}. We will focus on quadratic regularization, that is $\reg{\nObs}(\wVect) = \| \wVect \|^2_2$.

\paragraph{When $\ipm_\ipmFctClass = \mmd_\kernel$.} For these choices of integral probability metric and regularizer the optimization program from Definition \ref{def:cbdmIPM} is actually a quadratic program. This comes from the following expression \cite{gretton2007kernel,gretton2012kernel}:
\begin{align*}
   &\mmd_{\kernel}^2(\wDistr, \empTar) = \sum_{i,j = 1}^{\nObs} \w_i \w_j \kernel(\z_i, \z_j)  \\ &- 2 \sum_{i=1}^\nObs \w_i \mathbb{E}_{\joinVar \sim \empTar}[\kernel(\z_i,\joinVar)] + \mathbb{E}_{\joinVar,\joinVar' \sim \empTar^2}[\kernel(\joinVar,\joinVar')].
\end{align*}
See also equation \eqref{eq:MMDextension} in the Appendix. If, in the regression step (Step 2), $\regClass$ is taken to be the unit ball of a RKHS with kernel $\kernel_g$, then in view of condition (a) from Theorem \ref{thm:doubleConsistency}, it is natural to choose the following kernel:
\begin{equation} \label{eq:choiceKernel}
    \kernel(z, z') = 4 \kernel_{\responseFct}(z, z') \kernel_g(t,t') + \kernel_g(t,t')^2,
\end{equation}
where $z \doteq (t,x)$, and $\kernel_{\responseFct}$ is a kernel chosen such that its RKHS potentially approximates well $\responseFct$. Typically, if no particular assumption beside continuity can be made about $\responseFct$, then one should take $\kernel_{\responseFct}$ to be a universal kernel.
\paragraph{When $\ipm_\ipmFctClass = \wass$.} In fact, also for the Wasserstein distance $\wass$ the optimization program from Definition \ref{def:cbdmIPM} is a quadratic program (see Section \ref{sec:appendixAlgo} in the Appendix). However, when there is no clipping and $\regHyp = 0$, the weights are simply obtain as follows \cite{cuturi2014fast}:
\begin{equation*}
\forall i, \quad \w_i = \sum_{\substack{\z \in \text{Supp}(\empTar): \\ i = \argmin_{i}\{ d( \joinVar_i , \z )\}}} \empTar(\{ \z \}),
\end{equation*}
where we assume the $\argmin$ under the sum is unique; if not, one can split the mass uniformly over all the possible minimizers. We use this version in our experiments.

\paragraph{How to Tune $\kernel$.} Following \cite{kallus2016generalized,kallus2019kernel} we can derive approaches for hyperparameter tuning through a Bayesian interpretation of the bound in Theorem \ref{thm:causalBound}. That is, let's put a prior distribution on $\responseFct$, more precisely $\responseFct \sim \mathcal{GP}(m,\kernel_{\responseFct})$ is a Gaussian Process with mean function $m$ and kernel $\kernel_{\responseFct}$. For a fixed realization of $\responseFct$ the bound in Theorem \ref{thm:causalBound} holds for $\ipm_\ipmFctClass = \mmd_\kernel$ with $\kernel = 4 \responseFct^2 \kernel_g + \kernel_g^2$, and taking the expectation w.r.t.~$\responseFct$:
$$
\mathbb{E}_{\responseFct}[\causalRisk(\regFctMin)] \leq  \mathbb{E}_{\responseFct}[\inf_{\regFct \in \regClass} \causalRisk(\regFct)]+ \mmd_{\kernel}(\wDistr,\txTarget), 
$$
where we omitted the remaining terms for simplicity and where $\kernel =4 (m^2 + \kernel_\responseFct)\kernel_g + \kernel_g^2$. When $m=0$ we fall back to \eqref{eq:choiceKernel}. We consider Automatic Relevance Detection (ARD), that is we fit the following parametrized version of $\kernel_\responseFct$: $\xi \kernel_\responseFct(z / \boldsymbol{\theta}, z' / \boldsymbol{\theta})$, where $\xi$ and $\boldsymbol{\theta}$ are estimated by maximizing the marginal likelihood (see \cite{williams2006gaussian} for details). Moreover, we follow the Bayesian interpretation even further than \cite{kallus2016generalized,kallus2019kernel} as we also use the posterior distribution of $\responseFct$: we take $30$\% of the data on which we get a posterior distribution $\mathcal{GP}(m',\kernel'_{\responseFct})$ for $\responseFct$ (again see \cite{williams2006gaussian}), and we use $(m', \kernel'_\responseFct)$ instead of $(m, \kernel_\responseFct)$ to derive weights on the remaining data.

\paragraph{How to Select $\regHyp$ and $\clip$.} The hyperparameters $\regHyp$ and $\clip$ control for the variance that could be induced by weights that are too high. The bias on the other hand is controlled by the value of $\ipm_{\ipmFctClass}(\wDistr, \empTar)$, taken at the \cbdm~weights $\wVect$ (Def.~\ref{def:cbdmIPM}). We suggest to try different values for $\regHyp$ and $\clip$ and choose such parameters that achieve a good trade-off between a low $\ipm_\ipmFctClass$ value and a high Empirical Sample Size \cite{kong1994sequential}: $\text{ESS} \doteq 1 / \| \wVect \|_2^2$.

\section{EXPERIMENTS}

We compare our \cbdm~weighting methods, whose implementations are described in the previous section, with other weighting approaches for continuous treatments that are directly available from the literature. Namely, we compare with the algorithms CBGPS and npCBGPS from \cite{fong2018covariate} and with the GBM algorithm from \cite{zhu2015boosting}, which is an inverse probability weighting approach where the propensity score is fitted with gradient boosted models. We also plot results for when the regression step is done directly on the \emph{unweighted} data. On the other hand, the versions of \cbdm~that we will test are: \emph{\cbdm-Wass} which corresponds to the case where $\ipm_\ipmFctClass = \wass$ in Section \ref{sec:implementation} ; \emph{\cbdm-Poly4}, \emph{\cbdm-Gauss} and \emph{\cbdm-Exp} which correspond to $\ipm_\ipmFctClass = \mmd_\kernel$, where we chose the kernel $\kernel_\responseFct$ to be respectively a polynomial kernel of degree 4, the Gaussian kernel and the exponential kernel.

\begin{figure*}[h]

\includegraphics[width=\textwidth, height = 9cm]{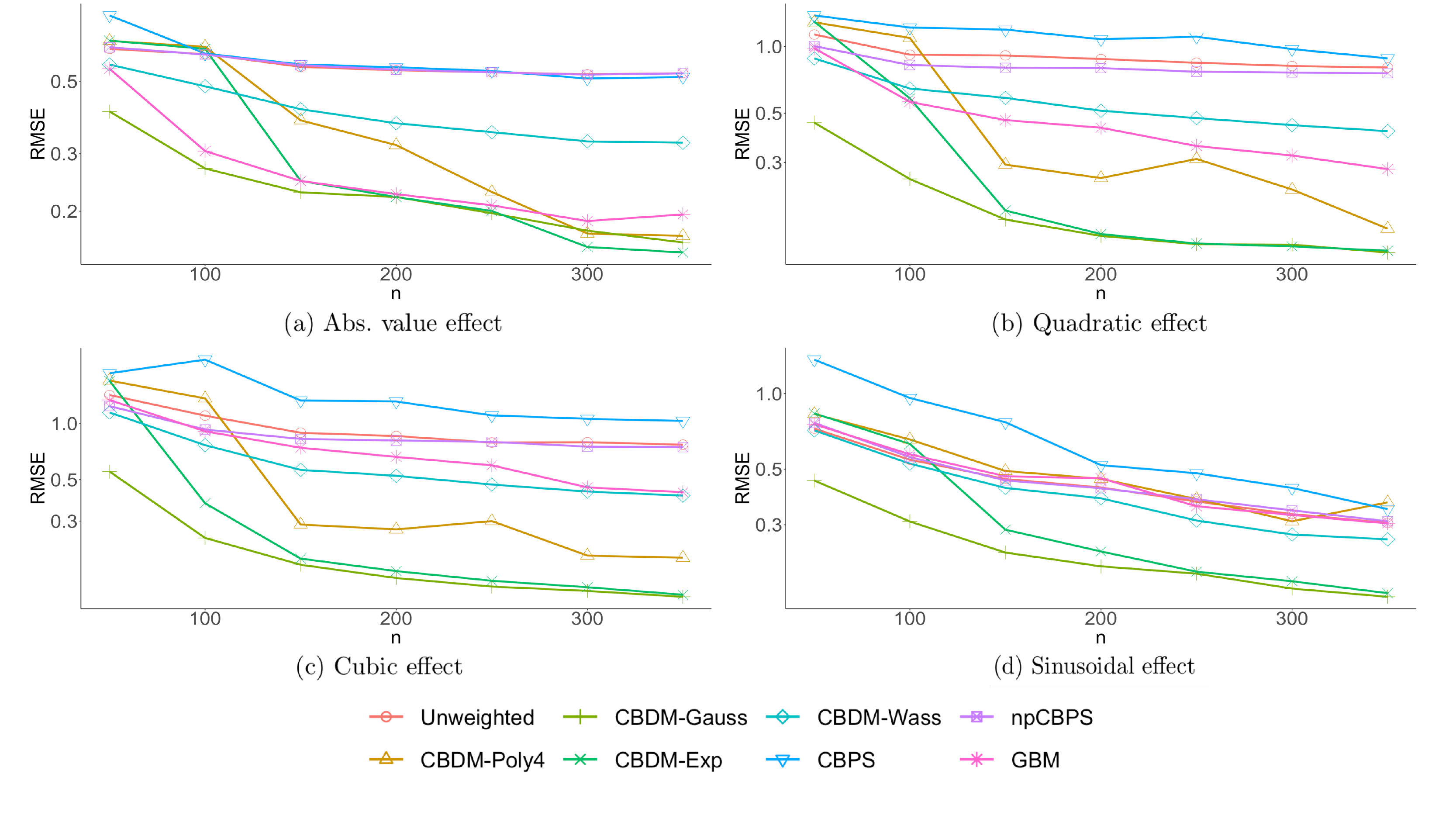}

\caption{RMSE over 100 replications of the estimated $\beta$ (in log-scale) as a function of the data size $\nObs$.} \label{fig:simulations}
\end{figure*}

\paragraph{Data Generating Process.} Our simulation setting mimics the one from \cite{kallus2017framework}, except we adapted it to a continuous treatment. More precisely, we let $ \cov = (\cov_1, \cov_2,\cov_3,\cov_4)$ four covariates uniformly distributed on $[-1,1]$ generated from a Gaussian copula with correlation $0.2$. The treatment $\treat$ conditioned on $\cov$ is distributed as: $T|X \sim \text{Beta}(5 m(X), 5(1-m(X)))$, where $m(X) = 0.8 / (1 + \sqrt{2} \| (X_1,X_2,X_3) \|_2)$. Also, we have:
$$
\out(\tr) = \beta \tr + f(\cov) + \epsilon, \quad \epsilon \sim \mathcal{N}(0,0.1),
$$
with $\beta = 1$ and $f$ which is composed of either: i) \emph{absolute values}, $f(x) = |x_1| + |x_2|$ ; ii) \emph{quadratic} polynomials, $f(x) = (x_1 + x_2) + (x_1 + x_2)^2$  ; iii) \emph{cubic} polynomials, $f(x) = (x_1 + x_2)^2 + (x_1 + x_2)^3$; or iv) \emph{sinusoidal} functions, $f(x) = \sin{(\pi(x_1 + x_2))} + \cos{(\pi(x_1 - x_2))}$. As the goal is to estimate $\beta$, we obviously choose $\kernel_g$ to be a polynomial kernel of degree 1. We set $W = 5$, $\lambda = 0$ and, for creating $\empTar$, we reshuffle the data several times until reaching around  $10^5$ samples in $\empTar$.

\paragraph{Results.} The results are displayed in Figure \ref{fig:simulations}. We see that even in this simple setting methods like CBGPS or npCBGPS seem to fail to improve upon \emph{unweighted} in most scenarios. Section \ref{sec:npCBPS} and Theorem \ref{thm:causalBound} provide some insights on why this is the case: By focusing only on the covariance between $\treat$ and $\cov$ the CBGPS methods aim mainly at controlling for linear functions $f$, however in the present settings $f$ is far from being linear. On the contrary, the \cbdm~methods with universal kernels, like the Gaussian or the exponential kernels, clearly appear to outperform other methods (or do at least as well) across all scenarios, and furthermore they converge faster as they almost reach their minimal errors often around $150$ observations. This suggests indeed that our advice from Section \ref{sec:implementation} of taking a universal kernel for $\kernel_\responseFct$ seems to be pertinent.

\section{CONCLUSION} We established in this paper another connection between transfer learning and causal inference as we showed that the balancing weight approach to treatment effect estimation can be restated as a discrepancy minimization problem from domain adaptation. Based on such an interpretation, we introduced a framework (\cbdm) that provably encompasses most of the weighting methods that directly optimize the balance between treatment and covariates, and allows to extend formally such approaches to treatments of arbitrary types. Furthermore, we derived theoretical guarantees for \cbdm~that are the generalizations to non-binary settings of known results for binary treatments. Such guarantees also offer more insight on how to choose hyperparameters (such as kernels) in order to achieve good performance in estimating a dose response-curve with observational data, as demonstrated in our simulations. We hope the connections made in this paper will open new perspectives, in particular we leave as an open question for future research the use of discrepancy notions other than \ipm s, but also more generally the use of other methods from domain adaptation for the purpose of dose-response curve estimation from observational data.

\newpage 

%\AtNextBibliography{\footnotesize}
%\printbibliographyss
%\bibliographystyle{plain}
%\bibliography{refs.bib}

\bibliographystyle{plain}
{\footnotesize \bibliography{bibliography.bib}}

\newpage 

\onecolumn

\renewcommand\thesection{\Alph{section}}

\setcounter{section}{0}

\begin{center}
    \textbf{\LARGE Appendix}
\end{center}

\section{Proofs of the Propositions from Section \ref{sec:special_cases}}\label{sec:appRem}

We present in this part of the appendix the proofs of Propositions \ref{prop1} and \ref{prop2} from Section \ref{sec:special_cases}.

\subsubsection*{Proof of Proposition \ref{prop1}} 

For a weight vector $\wVect$ on the full sample $(\treat_i,\cov_i)_{i = 1}^\nObs$, we divide it into two parts, that is $\wVect = (\wVect^0,\wVect^1)$ where $\wVect^0$ are the $\nObs_0 > 0$ first weights, and $\wVect^1$ are the last $\nObs_1 = \nObs - \nObs_0 > 0$ weights. Recall that we assumed that w.l.o.g.~the $\nObs_0$ first observations are the control group and the last $\nObs_1$ are the treatment group. First let's start the proof by noticing that (by definition of the empirical distribution):
$$
\forall k \in \{1, \ldots, K\}, \forall t \in \{ 0,1 \}, \quad \mathbb{E}_{\empTar}[f^t_k(\treat,\cov)] = \frac{n_t}{n} \bar{f}_k \quad \text{and} \quad \mathbb{E}_{\wDistr}[f^t_k(\treat,\cov)] = \sum_{i = 1}^{n_t} w_i^t f_k(\cov_{i + t . \nObs_0}).
$$
Now assume that $\wVect$ is actually the \cbdm~weights described in Proposition \ref{prop1} for some $\lambda > 0$. Let $\delta$ such that $\ipm_{\ipmFctClass}(\wDistr, \empTar) = \delta$. Notice because $\wVect \in \Lambda_{0,1}$, we have that $\nObs \wVect^t / \nObs_\tr \in \Lambda_{n_t}$ for $t \in \{ 0,1\}$. Also we have:
$$
\forall k \in \{1, \ldots, K\}, \forall t \in \{ 0,1 \}, \quad  |\nObs_t^{-1} \sum_{i=1}^{\nObs_t} \nObs \w^t_i f_k(X_{i + \tr. \nObs_{0}}) - \Bar{f}_k | \leq n \delta / n_t.
$$
Hence $\wVect^t$ is feasible for the problems \eqref{eq:approxExactBal} for $\delta_t = n \delta/ n_t$. It is also the one with minimum $L^2$ norm, because of the regularization term $\lambda \|\wVect\|_2^2$ for the \cbdm~program in Proposition \ref{prop1}. This implies the result of the proposition, namely $\wVect^t(\lambda) = \Tilde{\wVect}^t( \delta_\tr)$. 
\subsubsection*{Proof of Proposition \ref{prop2}} 
For the same reasons than in the proof of Proposition \ref{prop1}, we have here:
$$
\forall f \in \ipmFctClass_0, \forall t \in \{ 0,1 \}, \quad \mathbb{E}_{\empTar}[f^t(\treat,\cov)] = \frac{n_t}{n} n_1^{-1} \sum_{i = n_0 + 1}^{n} f(\cov_i) \quad \text{and} \quad \mathbb{E}_{\wDistr}[f^t(\treat,\cov)] = \sum_{i = 1}^{n_t} w_i^t f(\cov_{i + t . \nObs_0}).
$$
Therefore, we can rewrite the optimization problem the \cbdm~weights solve in this context as follows:
\begin{equation} \label{eq:prop2_rewrite}
    \wVect(\regHyp) =  \argmin_{\substack{\wVect = (\wVect^0,\wVect^1): \\ \nObs \wVect^0 / \nObs_0 \in \Lambda_{\nObs_0}, \,\, \nObs \wVect^1 / \nObs_1 \in \Lambda_{\nObs_1}}}  \regHyp \| \wVect^0 \|_2^2  + \regHyp \| \wVect^1 \|_2^2 + \max \left( \Delta_0(\wVect), \Delta_1(\wVect) \right),
\end{equation}
where $\Delta_\tr(\wVect) = \sup_{\fct \in \ipmFctClass_0}\left(\mathbb{E}_{\empTar}[f^t(\treat,\cov)] - \mathbb{E}_{\wDistr}[f^t(\treat,\cov)] \right)^2$. Notice also that actually $\Delta_\tr(\wVect)$ depends only on $\wVect^\tr$, that is: $\Delta_0(\wVect) = \Delta_0(\wVect^0)$ and $\Delta_1(\wVect) = \Delta_1(\wVect^1)$. First we can rewrite $\Delta_1(\wVect^1)$ as follows:
$$
\Delta_1(\wVect^1) = \sup_{\fct \in \ipmFctClass_0}\left(\frac{1}{n} \sum_{i = 1}^{n_1} f(\cov_{i+ \nObs_0}) - \sum_{i = 1}^{n_1} w_i^1 f(\cov_{i +  \nObs_0}) \right)^2,
$$
hence it is easy to see that the solution $\wVect^1(\regHyp)$ for \eqref{eq:prop2_rewrite} is $\wVect^1(\regHyp) = \nObs^{-1} \mathbf{1}$. So essentially the problem form \eqref{eq:prop2_rewrite} depends only on $\wVect^0$ and $\Delta_0$. Notice finally that by making the change of variable $ \Tilde{\wVect} = n \wVect^0 / \nObs_{0}$:
\begin{align*}
    \Delta_0(\wVect^0) + \regHyp \| \wVect^0 \|_2^2 & = \sup_{\fct \in \ipmFctClass_0}\left(\frac{n_0}{n} n_1^{-1} \sum_{i = 1}^{n_1} f(\cov_{i+ \nObs_0}) - \sum_{i = 1}^{n_0} w_i^0 f(\cov_{i}) \right)^2 + \regHyp \| \wVect^0 \|_2^2 \\
    & = (\nObs_{0} / n)^2 . \left( \ipm^2_{\ipmFctClass_{0}}(\hat{\source}_{\cov|\treat = 0}^{\Tilde{\wVect}},\hat{\source}_{\cov|\treat = 1}) + \regHyp \| \Tilde{\wVect} \|_2^2 \right).
\end{align*}

This actually concludes our proof as we can see now that $\wVect^0(\regHyp) = \frac{n_0}{n} \Tilde{\wVect}^0(\regHyp)$.

\section{Proof Of Theorem \ref{thm:causalBound}}

As a first step we give a bound on $\causalRisk(\regFctMin) - \causalRisk(\regFct)$, for any $\regFct \in \regClass$ and $\regFctMin \in \regClass$ that depends on the observational data $(\treat_i, \cov_i, \out_i)_{i = 1}^{\nObs}$. {\bf We use all along the same notations and assumptions as described in Theorem \ref{thm:causalBound}}.
\begin{lemma} \label{lem:firstStepBound} Consider the observational data $(\treat_i, \cov_i, \out_i)_{i = 1}^{\nObs}$ and any weight vector $\wVect \in \simplex$, we have:
$$
\forall \regFct \in \regClass, \quad \causalRisk(\regFctMin) - \causalRisk(\regFct) \leq \epsilon + \ipmCst \ipm_{\ipmFctClass}(\wDistr, \txTarget) + 2 \sup_{\regFct' \in \regClass} \sum_{i = 1}^{\nObs} \w_i (\out_i - \responseFct(\treat_i, \cov_i)) (\regFct'(\treat_i) - \regFct(\treat_i)).
$$
\end{lemma}
\begin{proof}
We use a bias-variance decomposition to see that we can actually replace $\epo$ by $\responseFct$. But first recall that:
\begin{align*}
    \causalRisk(\regFctMin) - \causalRisk(\regFct) = & \mathbb{E}_{\target_{\treat}}[(\epo(\treat) - \regFctMin(\treat))^2] - \mathbb{E}_{\target_{\treat}}[(\epo(\treat) - \regFct(\treat))^2] \\
    = & \mathbb{E}_{\target_{\treat}}[( \mathbb{E}_{\txTarget}[\condExp(\treat,\cov) |\treat] - \regFctMin(\treat))^2] - \mathbb{E}_{\target_{\treat}}[(\mathbb{E}_{\txTarget}[\condExp(\treat,\cov) |\treat] - \regFct(\treat))^2],
\end{align*}
where the last equality is a consequence of equation \eqref{eq:epoEquality}, tower property, and the fact that $\target_{\out|\treat,\cov} = \source_{\out|\treat,\cov}$. For any $\regFct \in \regClass$ (and in particular for $\regFct = \regFctMin$ too), by bias-variance decomposition we have:
$$
\mathbb{E}_{\target_{\treat}}[(\mathbb{E}_{\txTarget}[\condExp(\treat,\cov) |\treat] - \regFct(\treat))^2] = -\mathbb{E}_{\txTarget}[(\mathbb{E}_{\txTarget}[\condExp(\treat,\cov) |\treat] - \condExp(\treat,\cov))^2] + \mathbb{E}_{\txTarget}[( \condExp(\treat,\cov) - \regFct(\treat))^2].
$$
The first term on the RHS is actually common to all $\regFct \in \regClass$ and therefore will cancel out. Meaning that we get:
\begin{align*}
    \causalRisk(\regFctMin) - \causalRisk(\regFct) = & \mathbb{E}_{\txTarget}[(\condExp(\treat,\cov) - \regFctMin(\treat))^2] - \mathbb{E}_{\txTarget}[(\condExp(\treat,\cov) - \regFct(\treat))^2] \\
    = & \mathbb{E}_{\txTarget}[-2 \condExp(\treat,\cov)\regFctMin(\treat) + \regFctMin(\treat)^2] - \mathbb{E}_{\txTarget}[(\condExp(\treat,\cov) - \regFct(\treat))^2].
\end{align*}
Now, recall that by assumption for a constant we have $\forall \regFct \in \regClass$, $ - 2\condExp \regFct + \regFct^2$ is $(\gamma / 2, \epsilon / 2)$-approximable by $\ipmFctClass$. Using Definitions \ref{def:IPM} and \ref{def:approxFct} (and by adding and subtracting $\sum_i \w_{i} \condExp^2(\treat_i,\cov_i)$) directly yields:
$$
\causalRisk(\regFctMin) - \causalRisk(\regFct) \leq \epsilon + \ipmCst \ipm_{\ipmFctClass} (\wDistr, \txTarget) + \sum_{i = 1}^{\nObs} \w_{i} (\condExp(\treat_i,\cov_i)- \regFctMin(\treat_i))^2 - \sum_{i = 1}^{\nObs} \w_{i} (\condExp(\treat_i,\cov_i)- \regFct(\treat_i))^2.
$$
Again another bias-variance decomposition applies, $\forall \regFct \in \regClass$ we have:
$$
\mathbb{E}_{\out|\treat, \cov}\left[ \sum_{i = 1}^{\nObs} \w_i (\out_i - \regFct(\treat_i))^2 \right] = \mathbb{E}_{\out|\treat, \cov}\left[ \sum_{i = 1}^{\nObs} \w_i (\out_i - \condExp(\treat_i,\cov_i))^2 \right] + \sum_{i = 1}^{\nObs} \w_i (\condExp(\treat_i,\cov_i) - \regFct(\treat_i))^2,
$$
where for brevity we used $\mathbb{E}_{\out|\treat, \cov}$ to signify that we take the expectation over all the $\out_i$'s conditioned on the data $(\treat_i,\cov_i)_{i = 1}^{\nObs}$. We also drop the $\target$ because this conditional distribution $\out|\treat,\cov$ is the same under both $\source$ and $\target$. As the first term in the RHS is identical for all $\regFct \in \regClass$, it cancels out and we get:
\begin{align*}
    \causalRisk(\regFctMin) - \causalRisk(\regFct) \leq & \,\epsilon + \ipmCst \ipm_{\ipmFctClass} (\wDistr, \txTarget) + \mathbb{E}_{\out|\treat, \cov}\left[ \sum_{i = 1}^{\nObs} \w_i (\out_i - \regFctMin(\treat_i))^2 \right] - \mathbb{E}_{\out|\treat, \cov} \left[ \sum_{i = 1}^{\nObs} \w_i (\out_i - \regFct(\treat_i))^2 \right] \\
    \leq & \, \epsilon + \ipmCst \ipm_{\ipmFctClass} (\wDistr, \txTarget) + \mathbb{E}_{\out|\treat, \cov}\left[ \sum_{i = 1}^{\nObs} \w_i (\out_i - \regFctMin(\treat_i))^2 \right] - \sum_{i = 1}^{\nObs} \w_i (\out_i - \regFctMin(\treat_i))^2  \\
    & + \sum_{i = 1}^{\nObs} \w_i (\out_i - \regFctMin(\treat_i))^2 - \mathbb{E}_{\out|\treat, \cov} \left[ \sum_{i = 1}^{\nObs} \w_i (\out_i - \regFct(\treat_i))^2 \right]\\
    \leq & \, \epsilon + \ipmCst \ipm_{\ipmFctClass} (\wDistr, \txTarget) + \left( \mathbb{E}_{\out|\treat, \cov}\left[ \sum_{i = 1}^{\nObs} \w_i (\out_i - \regFctMin(\treat_i))^2 \right] - \sum_{i = 1}^{\nObs} \w_i (\out_i - \regFctMin(\treat_i))^2 \right)  \\
    & - \left( \mathbb{E}_{\out|\treat, \cov} \left[ \sum_{i = 1}^{\nObs} \w_i (\out_i - \regFct(\treat_i))^2 \right] - \sum_{i = 1}^{\nObs} \w_i (\out_i - \regFct(\treat_i))^2  \right),
\end{align*}
where for the last inequality we used that $\forall \regFct \in  \regClass,  \sum_{i = 1}^{\nObs} \w_i (\out_i - \regFct(\treat_i))^2 \geq  \sum_{i = 1}^{\nObs} \w_i (\out_i - \regFctMin(\treat_i))^2$, by definition of $\regFctMin$ which minimizes the weighted empirical risk (Step 2). For any $\regFct \in \regClass$ let's look at:
\begin{align*}
    \mathbb{E}_{\out|\treat, \cov} \left[ \sum_{i = 1}^{\nObs} \w_i (\out_i - \regFct(\treat_i))^2 \right] - \sum_{i = 1}^{\nObs} \w_i (\out_i - \regFct(\treat_i))^2 = & \sum_{i = 1}^{\nObs} \w_i \left( \mathbb{E}[\out_i^2|\treat_i, \cov_i] - 2 \regFct(\treat_i) \condExp(\treat_i, \cov_i) + \regFct^2(\treat_i) \right) \\
    & - \sum_{i = 1}^{\nObs} \w_i \left( \out_i^2 - 2 \regFct(\treat_i) \out_i + \regFct^2(\treat_i) \right) \\
    = & 2 \sum_{i = 1}^{\nObs} \w_i (\out_i - \condExp(\treat_i, \cov_i)) . \regFct(\treat_i) - \sum_{i = 1}^{\nObs} \w_i (\out_i^2 - \mathbb{E}[\out_i^2|\treat_i, \cov_i]).
\end{align*}
The last term being identical for any $\regFct \in \regClass$, it cancels out: Plugging the RHS of the last equality in the previous inequality yields:
$$
\forall \regFct \in \regClass, \quad \causalRisk(\regFctMin) - \causalRisk(\regFct) \leq \epsilon + \ipmCst \ipm_{\ipmFctClass}(\wDistr, \txTarget) + 2 \sum_{i = 1}^{\nObs} \w_i (\out_i - \condExp(\treat_i, \cov_i)) (\regFctMin(\treat_i) - \regFct(\treat_i)).
$$
As $\regFctMin \in \regClass$, we can bound this last inequality replacing $\regFctMin$ by a supremum over $\regClass$. This concludes our proof.
\end{proof}
The reminder of this proof aims at bounding the term $\sup_{\regFct' \in \regClass} \sum_{i = 1}^{\nObs} \w_i (\out_i - \condExp(\treat_i, \cov_i)) (\regFct'(\treat_i) - \regFct(\treat_i))$ from the previous lemma. First we bound its expectation (still conditioned on $(\treat_i,\cov_i)_{i = 1}^{\nObs}$) and then apply a concentration inequality. Recall that the Rademacher complexity $\radComp$ was introduced in Definition \ref{def:rademacher}.
\begin{lemma}
Use, for short, $\mathbb{E}_{\out|\treat, \cov}$ to signify a conditional expectation w.r.t.~$(\out_i)_{i = 1}^\nObs$ given $(\treat_i, \cov_i)_{i = 1}^\nObs$. We have:
$$
\forall \regFct \in \regClass, \quad \mathbb{E}_{\out|\treat, \cov} \left[\sup_{\regFct' \in \regClass} \sum_{i = 1}^{\nObs} \w_i (\out_i - \condExp(\treat_i, \cov_i)) (\regFct'(\treat_i) - \regFct(\treat_i)) \right] \leq \clip \yConstRange \radComp /2.
$$
\end{lemma}
\begin{proof}
First notice that because we first fixed $\regFct \in \regClass$, we have:
$$
\mathbb{E}_{\out|\treat, \cov} \left[\sup_{\regFct' \in \regClass} \sum_{i = 1}^{\nObs} \w_i (\out_i - \condExp(\treat_i, \cov_i)) (\regFct'(\treat_i) - \regFct(\treat_i)) \right] = \mathbb{E}_{\out|\treat, \cov} \left[\sup_{\regFct' \in \regClass} \sum_{i = 1}^{\nObs} \w_i (\out_i - \condExp(\treat_i, \cov_i)) \regFct'(\treat_i) \right].
$$
We can proceed by symmetrization: For every $i$, we introduce another variable $\out'_i$ which, conditionally on $\treat_i, \cov_i$, is independent of $\out_i$ but identically distributed. In particular, $\mathbb{E}[\out'_i|\treat_i, \cov_i] = \condExp(\treat_i, \cov_i)$, and we have:
$$
\mathbb{E}_{\out|\treat, \cov} \left[\sup_{\regFct' \in \regClass} \sum_{i = 1}^{\nObs} \w_i (\out_i - \condExp(\treat_i, \cov_i)) \regFct'(\treat_i) \right] \leq \mathbb{E}_{\out, \out'|\treat, \cov} \left[\sup_{\regFct' \in \regClass} \sum_{i = 1}^{\nObs} \w_i (\out_i - \out'_i) \regFct'(\treat_i) \right].
$$
Note that for any Rademacher variable $\radVar_i$ independent of the rest, because of the symmetry between $\out_i$ and $\out'_i$, we have:
$$
\radVar_i (\out_i - \out'_i) \sim (\out_i - \out'_i).
$$
Hence,
\begin{align*}
    \mathbb{E}_{\out, \out'|\treat, \cov} \left[\sup_{\regFct' \in \regClass} \sum_{i = 1}^{\nObs} \w_i (\out_i - \out'_i) \regFct'(\treat_i) \right] = & \mathbb{E}_{\out, \out'|\treat, \cov} \left[ \mathbb{E}_{\radVarVect} \left[ \sup_{\regFct' \in \regClass} \sum_{i = 1}^{\nObs} \w_i \radVar_i (\out_i - \out'_i) \regFct'(\treat_i) \right]\right] \\
    \leq & \clip \yConstRange \radComp /2,
\end{align*}
where the last inequality comes from the Contraction Lemma \cite[Lemma 26.9]{shalev2014understanding}, noticing that $\w_i \leq \clip / \nObs$ and $|\out_i - \out'_i| \leq \yConstRange$.
\end{proof}
To derive the concentration inequality, we will make use of the following theorem, taken from \cite{pinelis1994optimum}, which is a generalization of Bernstein's inequality for martingales.
\begin{theorem} [Theorem 8.7 from \cite{pinelis1994optimum}] \label{thm:pinelis}
Let $(X_j)_{j \in \mathbb{N}}$ a martingale taking values in $\mathbb{R}$, where $X_0 = 0$ and, for the same filtration $\mathcal{F} = (\mathcal{F}_{j})_{j = 1}^{\nObs}$, let $(u_j)_{j\in \mathbb{N}}$ any $\mathcal{F}$-adapted process such that $X_j - X_{j-1} = u_j - \mathbb{E}[u_j| \mathcal{F}_{j-1}]$ having the following properties:
\begin{itemize}
    \item $\exists a > 0$ s.t.~$\forall j, |u_j| \leq a$ a.s.
    \item $\exists b > 0$ s.t.~$ \sum_{j\geq 1} \mathbb{E}[u_j^2|\mathcal{F}_{j-1}] \leq b^2$ a.s.
\end{itemize}
Then, if we call $h(x) = (1+x)\log(1+x) - x$ we have:
$$
\forall \epsilon > 0, \quad \mathbb{P}(\sup_j X_j \geq \epsilon) \leq \exp\left(- \frac{b^2}{a^2} h\left(\frac{\epsilon a}{b^2}\right)\right).
$$
\end{theorem}
\begin{lemma}
For given (that is fixed) $(\treat_i, \cov_i)_{i=1}^{\nObs}$ we have for all $\nObs>0$ and $\delta \in (0,1]$, with probability at least $1-\delta$ over $(\out_i)_{i=1}^\nObs$:
\begin{align*}
    \sup_{\regFct' \in \regClass} \sum_{i = 1}^{\nObs} \w_i (\out_i - \condExp(\treat_i, \cov_i)) (\regFct'(\treat_i) - \regFct(\treat_i))  \leq & \mathbb{E}_{\out|\treat, \cov} \left[\sup_{\regFct' \in \regClass} \sum_{i = 1}^{\nObs} \w_i (\out_i - \condExp(\treat_i, \cov_i)) (\regFct'(\treat_i) - \regFct(\treat_i)) \right] \\
    & + \frac{4}{3} \frac{\clip \yConstRange \boundedGees \log(1/\delta)}{\nObs}+ 2\sigma \boundedGees\| \wVect \|_2 \sqrt{2 \log(1/\delta)}.
\end{align*}
\end{lemma}
\begin{proof}
We want to use the previous theorem to derive this concentration bound. Consider the following Doob's martingale:
$$
M_j = \mathbb{E}\left[ \sup_{\regFct' \in \regClass} \sum_{i = 1}^{\nObs} \w_i (\out_i - \condExp(\treat_i, \cov_i)) (\regFct'(\treat_i) - \regFct(\treat_i)) \Big| (\treat_i,\cov_i)_{i=1}^{\nObs}, (\out_i)_{i=1}^{j} \right], \quad \forall j \in \{0, \ldots, \nObs \}.
$$
By convention we will set $M_j = M_\nObs$ for $j > \nObs$, and note that we can subtract $M_j$ by $M_0$ so that it would be equal zero for $j = 0$ (as required in Theorem \ref{thm:pinelis}). Note also that $M_\nObs$ is the supremum on the LHS of this lemma's inequality and $M_0$ is the expectation on the RHS. Finally, let $u_j = M_j - M_{j-1}$ and notice that $\mathbb{E}[u_j|\mathcal{F}_{j-1}] = 0$, where we call $\mathcal{F}_{j} = \sigma((\treat_i,\cov_i)_{i=1}^{\nObs}, (\out_i)_{i=1}^{j})$. Moreover, through some tedious but straightforward manipulations we can easily see that:
$$
|u_j| \leq 2 \w_j |\out_j- \condExp(\treat_j,\cov_j)|\boundedGees \leq 2 \clip \yConstRange \boundedGees / \nObs = a,
$$
and also:
$$
\sum_{j=1}^{\nObs}\mathbb[u_j^2|\mathcal{F}_{j-1}] \leq \sum_{j=1}^{\nObs} 4 \w_j^2 \yVar \boundedGees^2 = b^2.
$$
Using Theorem \ref{thm:pinelis} we get that for any $\epsilon$:
$$
\mathbb{P}(M_\nObs - M_0 \geq \epsilon) \leq \exp \left(- \frac{b^2}{a^2} h\left(\frac{\epsilon a}{b^2}\right) \right) \leq \exp \left(-\frac{\epsilon^2 / 2}{b^2 + \epsilon a /3} \right),
$$
where the last inequality comes from the fact that $h(x) \geq x^2/(2+2x/3)$ (see lemmas B.8 and B.9 from \cite{shalev2014understanding}). Then equating this upper bound to $\delta$ and solving for $\epsilon$ (the steps follows the usual ones used for the classical Bernstein's inequality), we get:
$$
\epsilon \leq 2 \log(1/\delta) a /3 + \sqrt{2 \log(1/\delta) b^2}.
$$
Plugging in the values of $a$ and $b^2$ yields the result of this lemma.
\end{proof}
\subsubsection*{Proof of Theorem \ref{thm:causalBound}} 
Combining the three previous lemmas, we get that with probability at least $1-\delta$ over $(\out_i)_{i=1}^{\nObs}$, given $(\treat_i,\cov_i)_{i=1}^\nObs$ (being fixed):
$$
\forall \regFct \in \regClass, \quad \causalRisk(\regFctMin) - \causalRisk(\regFct) \leq \epsilon + \ipmCst \ipm_{\ipmFctClass}(\wDistr, \txTarget) + \clip \yConstRange \radComp + \frac{8}{3} \frac{\clip \yConstRange \boundedGees \log(1/\delta)}{\nObs}+ 4\sigma \boundedGees \| \wVect \|_2 \sqrt{2 \log(1/\delta)}.
$$
This concludes the proof of Theorem \ref{thm:causalBound}.

\section{Proof Of Theorem \ref{thm:dualProblem}} \label{app:dualPb}

We consider here the \cbdm~weights $\wVect(\nObs)$ from Definition \ref{def:cbdmIPM} built on data $(\joinVar_{i})_{i = 1}^{\nObs}$ with $\ipm_\ipmFctClass = \mmd_{\kernel}$, where $\kernel$ is a continuous kernel that gives rise to a RKHS $(\rkhs, \inner{\cdot}{\cdot})$. That is, here $\ipmFctClass$ is the unit ball of $\rkhs$: $\ipmFctClass = \mathcal{B}_{\rkhs}(0,1)$. The proof of Theorem \ref{thm:dualProblem} starts by noticing that $\nObs \w_i$ is actually the density $\partial \wDistr / \partial \hat{\source}^\nObs_{\joinVar}$ taken at $\joinVar_i$, where $\hat{\source}^\nObs_{\joinVar} = \sum_{i=1}^\nObs \dirac{\joinVar_i} /\nObs$. This allows us to rewrite the optimization program of Definition \ref{def:cbdmIPM} as follows:
\begin{equation} \label{eq:rewrite}
    \forall i, \w_i(\nObs) = \frac{\phi_0(\joinVar_i)}{\nObs} \quad \text{ with } \quad \phi_0 = \argmin_{\substack{\phi \in L^2(\hat{\source}^\nObs_{\joinVar}), \\ \int \phi(\z) d \hat{\source}^\nObs_{\joinVar} = 1}} \mmd^2_{\kernel}(\hat{\source}^{\nObs}_{\phi}, \empTar^\nObs) + \frac{\regHyp }{\nObs} \int \convRegFct(\phi(\z)) d \hat{\source}^\nObs_{\joinVar},
\end{equation}
where $\hat{\source}^{\nObs}_{\phi} = \sum_{i = 1}^{\nObs} \phi(\joinVar_i) \dirac{\joinVar_i} / \nObs$ and the function $\convRegFct(x)$ is strictly convex, continuous on $[0 , \clip]$, differentiable on $(0,\clip)$, and $\convRegFct(x) = + \infty$ when $x \notin [0,\clip]$, as defined in Theorem \ref{thm:dualProblem}. In the following, we will denote by $\Phi$ the feature mapping into $\rkhs$, that is: $\forall \z \in \txSupp, \, \Phi(\z) = \kernel(\z,\cdot) \in \rkhs$. Recall also that the Bochner integral, which is the extension of the Lebesgue integral for functions taking values in an Hilbert space (more generally in Banach spaces), can be interchanged with any continuous linear operator, in particular: $\forall \fct \in \rkhs, \,\langle \fct , \int \phi(z) \Phi(z) d \nu \rangle_\rkhs = \int \phi(z) \langle \fct ,  \Phi(z) \rangle_\rkhs  d \nu = \int \phi(z) \fct(z) d \nu$, for some measure $\nu$. Using this property and Cauchy-Schwarz inequality, we restate a known result about $\mmd$ \cite{gretton2007kernel,gretton2012kernel}:
\begin{align} \label{eq:rewriteMMD}
    \mmd^2_{\kernel}(\nu, \pi) &= \sup_{\fct \in \mathcal{B}_{\rkhs}(0,1)} \left( \int \fct(z) d \nu - \int  \fct(z) d \pi \right)^2  \nonumber\\
    & =\sup_{\fct \in \mathcal{B}_{\rkhs}(0,1)} \left \langle \fct , \int \Phi(z) d \nu - \int  \Phi(z) d \pi \right \rangle_\rkhs^2 \nonumber \\
    & = \left \| \int \Phi(z) d \nu - \int  \Phi(z) d \pi \right \|_\rkhs^2,
\end{align}
where we consider from now on $\nu$ and $\pi$ as being two probability measures. Problem \eqref{eq:rewrite} can then be written as a special case (taking $\nu = \hat{\source}^\nObs_{\joinVar}$ and $\pi = \empTar^\nObs$) of the following optimization program:
\begin{equation} \label{eq:finalPbVersion}
    \frac{\regHyp}{\nObs} \cdot \inf_{\substack{\phi \in L^2(\nu), \\ \int \phi(\z) d \nu = 1}} \int \convRegFct(\phi(\z)) d \nu + \frac{\nObs}{\regHyp} \left \| \int \phi(z) \Phi(z) d \nu - \mathbf{a} \right \|_\rkhs^2 ,
\end{equation}
where $\mathbf{a} = \int \Phi(\z) d \pi$. Using words, this program corresponds to the problem of finding a distribution with density $\phi$ w.r.t.~$\nu$ (which is bounded by $\clip$ as $\rho$ is infinite outside $[0,\clip]$) such that it leads to an expectation of $\Phi(Z)$ in $\rkhs$ close enough to $\mathbf{a} \in \rkhs$, while keeping the penalty depending on $\rho$ as low as possible. Namely, if $\rho(x) = x^2$ on $[0,\clip]$, then it seeks to find such density with a low L2 norm. When $\rho(x) = x \log x$ on $[0,\clip]$, this becomes the problem of finding such distribution with a large (differential) entropy.  Finally, let's recall also that, for any convex function $\rho$ defined on an Hilbert space with inner product $\langle \cdot , \cdot \rangle$, its Legendre transform, denoted $\rho^*$, is the convex function defined as:
$$
\rho^*(x) = \sup_{y} (\langle x , y \rangle - \rho(y)).
$$
We will make good use of the following properties of the Legendre transform of $\rho$.
\begin{lemma} \label{lem:legTransRho}
Let $\rho^*(x) = \sup_{y \in \mathbb{R}} x y - \rho(y), \forall x \in \mathbb{R}$, the Legendre transform of $\rho$ from Theorem \ref{thm:dualProblem}. Then:
\begin{itemize}
    \item  $\rho^*$ is finite and differentiable on $\mathbb{R}$, and its differential is bounded. More precisely, $\forall x \in \mathbb{R}, \frac{d\rho^*}{d x}(x) \in [0, \clip]$.
    \item Let $\frac{d\rho}{d x}(0)$ and $\frac{d\rho}{d x}(\clip)$ the limits (possibly infinite) of the differential of $\rho$ respectively at $0$ (from above) and at $\clip$ (from below). Then $\rho^*(x) = -\rho(0)$ for $x \in (-\infty, \frac{d\rho}{d x}(0)]$; $\rho^*(x) = \clip x - \rho(\clip)$ for $x \in [\frac{d\rho}{d x}(\clip), +\infty)$; and $\rho^*$ is strictly convex on $(\frac{d\rho}{d x}(0),\frac{d\rho}{d x}(\clip))$.
\end{itemize}
\end{lemma}
\begin{proof}
Note that, from definition of $\rho$, we have $\rho^*(x) = \sup_{y \in [0,\clip]} x y - \rho(y)$. $\rho$ being continuous and strictly convex on $[0,\clip]$, we have $\forall x \in \mathbb{R}, \exists ! \,y_x \in [0,\clip] \,\, s.t. \,\, \rho^*(x) = x y_x - \rho(y_x)$. Recall a well-known property [see \cite{villani2003topics}, Proposition 2.4] that is actually true for any proper lower semi-continuous convex function like $\rho$:
\begin{equation}\label{eq:convEqual}
    \forall x,y \in \mathbb{R}, \quad x y = \rho(y) + \rho^*(x)  \Longleftrightarrow y \in \partial \rho^*(x) \Longleftrightarrow x \in \partial \rho(y),
\end{equation}
where $\partial$ refers to the sub-differential. The existence and unicity of $y_x$ imply that $\forall x \in \mathbb{R}, \partial \rho^*(x) = \{ y_x \}$, showing the first point of the lemma.

For the second point, notice that by first order condition $y_x \in (0, \clip) \Longleftrightarrow x \in (\frac{d\rho}{d x}(0),\frac{d\rho}{d x}(\clip))$. Hence, this explains the expression of $\rho^*$ outside $(\frac{d\rho}{d x}(0),\frac{d\rho}{d x}(\clip))$. Also, because $\rho$ is strictly convex and differentiable on $(0, \clip)$, then $\frac{d\rho}{d x}$ is strictly increasing and invertible on $(0,\clip)$. This means that $y_x = \frac{d\rho}{d x}^{-1}(x)$ is also strictly increasing w.r.t.~$x$ on $(\frac{d\rho}{d x}(0),\frac{d\rho}{d x}(\clip))$. $\frac{d\rho^*}{d x}(x)$ being equal to $y_x$, we have therefore the strict convexity of $\rho^*$ on $(\frac{d\rho}{d x}(0),\frac{d\rho}{d x}(\clip))$.
\end{proof}

It is interesting to look at what $\rho^*$ and $\frac{d\rho^*}{d x}$ looks like for some common regularizers.

\paragraph{For quadratic regularization.} When $\rho(x) = x^2$ for $x \in [0, \clip]$ and $\rho(x) = + \infty$ elsewhere, the Legendre transform of $\rho$ is:
$$
\rho^*(x) = \sup_{y \in [0,\clip]} x y - y^2 = \left\{ \begin{array}{cc}
      0 & \text{ if } x\leq 0,  \\
      x^2/4 & \text{ if } x \in (0 , 2\clip),\\
     \clip x - \clip^2 & \text{ otherwise}.
\end{array}\right.
$$
Notice also that its derivative is $1/2$-Lipschitz, more precisely:
$$
\frac{d\rho^*}{d x}(x) = \left\{ \begin{array}{cc}
      0 & \text{ if } x\leq 0,  \\
      x/2 & \text{ if } x \in (0 , 2\clip),\\
     \clip & \text{ otherwise}.
\end{array}\right.
$$

\paragraph{For entropic regularization.} When $\rho(x) = x \log x$ for $x \in [0, \clip]$ (equals $0$ at $0$ of course) and $\rho(x) = + \infty$ elsewhere, the Legendre transform of $\rho$ is:
$$
\rho^*(x) = \sup_{y \in [0,\clip]} x y - y^2 = \left\{ \begin{array}{cc}
      \exp(x - 1) & \text{ if } x \in (-\infty , 1 + \log \clip),\\
     \clip x - \clip \log(\clip) & \text{ otherwise}.
\end{array}\right.
$$
And its derivative is:
$$
\frac{d\rho^*}{d x}(x) = \left\{ \begin{array}{cc}
     \exp(x - 1) & \text{ if } x \in (-\infty , 1 + \log \clip),\\
     \clip & \text{ otherwise}.
\end{array}\right.
$$

To analyse program \eqref{eq:finalPbVersion} we are going to make use of the well-known Fenchel-Rockafellar theorem, of which we give below an adaptation to our setting of the version from \cite{villani2003topics} (Theorem 1.9). Recall that $L^2(\nu)$ is an Hilbert space and that, by Riesz representation theorem, any Hilbert space can be identified with its own dual.
\begin{theorem}[Adaptation of Fenchel-Rockafellar Theorem] \label{thm:fenchRock} Let $(E,\langle ., . \rangle)$ be an Hilbert space and $\Theta$, $\Xi$ two convex functions on $E$ that could be equal to $+\infty$ for some $x \in E$. Let $\Theta^*$ and $\Xi^*$ their Legendre transforms. If there exists $x_0 \in E$ such that:
$$
\Theta(x_0) < + \infty, \quad \Xi(x_0) < +\infty, \quad \text{and} \quad \Theta \text{ is continuous at } x_0.
$$
Then,
$$
\inf_{x \in E} \Theta(x) + \Xi(x) = - \min_{x \in E} \Theta^*(x) + \Xi^*(-x).
$$
The $\min$ on the RHS signifies, of course, that the optimum is achieved by some $x \in E$.
\end{theorem}
 
Using Theorem \ref{thm:fenchRock} we can now derive the following dual formulation of problem \ref{eq:finalPbVersion}:
 
\begin{lemma}\label{lem:duality}
The minimization problem of \eqref{eq:finalPbVersion}, when not being equal to $+\infty$, admits a $\nu$-a.s.~unique solution (hence the $\inf$ has been changed to a $\min$ below) and has the following dual formulation:
$$
\min_{\substack{\phi \in L^2(\nu), \\ \int \phi(\z) d \nu = 1}} \int \convRegFct(\phi(\z)) d \nu + \frac{\nObs}{\regHyp} \left \| \int \phi(z) \Phi(z) d \nu - \mathbf{a} \right \|_\rkhs^2 = - \inf_{\mu \in \mathbb{R}, \fct \in \rkhs } \int \rho^*\left(- \mu - \fct (z) \right) d \nu + \mu + \inner{\fct}{\mathbf{a}} + \frac{\regHyp}{4 \nObs} \| f \|^2_{\rkhs}.
$$
 Furthermore, if $(\mu^0,\fct^0)$ is any optimal solution of the RHS, then $\phi_0(z) = \frac{d\rho^*}{d x}(- \mu^0 -  \fct^0 (z))$ is the $\nu$-a.s.~unique solution of the LHS, that is of \eqref{eq:finalPbVersion}.
\end{lemma}
\begin{proof}
We are going to use Fenchel-Rockafellar theorem starting with the RHS problem as the regularity conditions required by the theorem are verified for it. This would also imply the existence of the optimum of the LHS, the $\nu$-a.s.~unicity coming from the fact that $\rho$ is strictly convex on its domain. Define for any $\phi \in L^2(\nu)$:
\begin{align}\label{eq:thetaXiFct}
    \Theta(\phi) & = \int \rho^*(\phi(z)) d\nu, \\
    \Xi(\phi) & = \left\{ \begin{array}{cc}
         & \mu +  \inner{\fct}{\mathbf{a}} + \frac{\regHyp}{4 \nObs} \| f \|^2_{\rkhs}, \quad \text{if} \quad \phi = - \mu - \fct, \, \text{with } \fct \in \rkhs, \\
         & +\infty \quad  \quad \quad \quad \quad \quad  \text{otherwise}.  \quad \quad \quad \quad \quad
    \end{array} \right.\nonumber
\end{align}
If $1 \in \rkhs$, then replace the condition $\fct \in \rkhs$ by $\fct$ belongs to the orthogonal complement of $\{ 1\}$ in $\rkhs$. Obviously $ - \text{RHS} = \inf_{\phi \in L^2(\nu)} \Theta(\phi) + \Xi(\phi)$. We can see also that the regularity conditions required by Fenchel-Rockafellar theorem are easily verified: e.g.~$\Theta(0) < + \infty$ and $\Xi(0) < + \infty$ ; the continuity of $\Theta$ is a consequence of Lemma \ref{lem:legTransRho} (indeed, as the derivative of $\rho^*$ is bounded, in absolute value, by $\clip$ it is actually $\clip$-Lipschitz, and so is $\Theta$ on $L^2(\nu)$). Let's look at the Legendre transforms.
\begin{equation} \label{eq:thetaLegendre}
    \Theta^*(\phi) = \sup_{\psi \in L^2(\nu)} \langle \phi , \psi \rangle_{L^2(\nu)} - \Theta(\psi) = \sup_{\psi \in L^2(\nu)} \int \phi(z) \psi(z) - \rho^*(\psi(z)) d \nu \leq  \int \rho(\phi(z)) d \nu,
\end{equation}
where the last inequality is just a direct consequence of the definition of the Legendre transform. As matter of fact, this inequality is an equality. For instance, for any $b>0$ construct $\psi_b \in L^2(\nu)$ such that:
$$
\psi_b(z) = \frac{d\rho}{d x} (\phi(z)) \text{ when } \phi(z) \in (0, \clip) \quad ; \quad \psi_b(z) = -b \text{ when } \phi(z) \leq 0 \quad ; \quad \psi_b(z) = b \text{ when } \phi(z) \geq \clip.
$$
The constant $b$ is just here to make sure that $\psi_b$ is really in $L^2(\nu)$. Using \eqref{eq:convEqual} and Lemma \ref{lem:legTransRho}, it can be shown that plugging $\psi_b$ in the supremum of equation \eqref{eq:thetaLegendre} and sending $b \rightarrow + \infty$ yields the sought equality. That is:
$$
\Theta^*(\phi) =  \int \rho(\phi(z)) d \nu.
$$
Now let's turn to the Legendre transform of $\Xi$:
$$
\Xi^*(\phi) = \sup_{\psi \in L^2(\nu)} \langle \phi , \psi \rangle_{L^2(\nu)} - \Xi(\psi) = \sup_{\mu, \fct} \mu \int (\phi(z) + 1) d \nu  + \inner{\fct}{\int (\phi(z) \Phi(z) + \mathbf{a}) d \nu}  - \frac{\regHyp}{4 \nObs} \| \fct \|^2_\rkhs \approxDelta.
$$
Note that from the first term in the R.H.S.~we would have $\Xi^*(\phi) = + \infty$ if $\int \phi(z) d \nu \neq -1$.  By Cauchy-Schwarz inequality, the remaining terms are bounded by $\|\fct\|_\rkhs \|\int (\phi(z) \Phi(z) + \mathbf{a}) d \nu\|_\rkhs  - \frac{\regHyp}{4 \nObs} \| \fct \|^2_\rkhs$, which is attained when $\fct \propto \int (\phi(z) \Phi(z) + \mathbf{a}) d \nu$. Optimizing w.r.t.~$\|\fct\|_\rkhs$ yields:
$$
\Xi^*(\phi) = \left\{ \begin{array}{cc}
     & \frac{\nObs}{\regHyp}\|\int (\phi(z) \Phi(z) + \mathbf{a}) d \nu\|^2_\rkhs \quad \text{ if } \int (\phi(z) + 1) d \nu = 0,\\
     & \quad\quad\quad \quad+\infty \quad\quad \quad\quad\quad\quad\text{otherwise.} \quad\quad\quad\quad\quad. 
\end{array}\right.
$$
Applying Fenchel-Rockafellar theorem yields the dual formulation of \eqref{eq:finalPbVersion}. Furthermore, let $(\mu^0,\fct^0)$ any optimal solution of the RHS and $\psi_0(z) = - \mu - \fct(z)$. Let also $\phi_0$ the optimum for the LHS problem. From our use of the Fenchel-Rockafellar theorem we can see that:
$$
\Theta(\phi_0) + \Theta^*(\psi_0) + \Xi(\phi_0) + \Xi^*(-\psi_0) = 0.
$$
As, again by definition of Legendre transform, $\Theta(\phi_0) + \Theta^*(\psi_0) \geq \langle \phi_0 , \psi_0 \rangle_{L^2(\nu)}$ and $\Xi(\phi_0) + \Xi^*(-\psi_0) \geq - \langle \phi_0 , \psi_0 \rangle_{L^2(\nu)}$, the above equality simply means that $\Theta(\phi_0) + \Theta^*(\psi_0) = \langle \phi_0 , \psi_0 \rangle_{L^2(\nu)}$, which can be rewritten as:
$$
\int \rho(\phi_0(x)) + \rho^*(\psi_0(x)) - \phi_0(x)\psi_0(x) d \nu(x) = 0.
$$
Once again, the definition of Legendre transform gives that $\rho(\phi_0(x)) + \rho^*(\psi_0(x)) - \phi_0(x)\psi_0(x) \geq 0$. Hence the function inside the above integral is equal to zero $\nu$-almost surely. Therefore, by using \eqref{eq:convEqual}, we conclude with the last result of this lemma.
\end{proof}

\subsubsection*{Proof of Theorem \ref{thm:dualProblem}} 

Theorem \ref{thm:dualProblem} is a consequence of the reformulation of the optimization program from Definition \ref{def:cbdmIPM} into \eqref{eq:rewrite} and \eqref{eq:finalPbVersion}, of the dual program derived in Lemma \ref{lem:duality} and of the Representer Theorem \cite{hofmann2008kernel, shalev2014understanding}: That is, when $\nu = \hat{\source}^\nObs_{\joinVar}$ and $\pi = \empTar^\nObs$ the solution $\fct^0$ of the dual problem from Lemma \ref{lem:duality} can be expressed as $\fct^0(\cdot) = \sum_{z \in \hat{\txSupp}} \alpha(z) \kernel(z,\cdot) = \boldsymbol{\alpha}^{T} \mathbf{k}(\cdot)$, where $\hat{\txSupp} = \text{supp}(\hat{\source}^\nObs_{\joinVar}) \bigcup \text{supp}(\empTar)$ and $\boldsymbol{\alpha} = (\alpha(z))_{z \in \hat{\txSupp}}$ and $\mathbf{k}(\cdot) = (\kernel(z,\cdot))_{z \in \hat{\txSupp}}$. Call also $\mathbf{K} = (\kernel(z,z'))_{z,z' \in \hat{\txSupp}}$, hence the dual problem can be rewritten as follows:
$$
\min_{\mu, \boldsymbol{\alpha}} \frac{1}{\nObs} \sum_i \rho^*\left(- \mu - \boldsymbol{\alpha}^{T} \mathbf{k}(\joinVar_i) \right) + \mu + \boldsymbol{\alpha}^T   \mathbb{E}_{\empTar}[\mathbf{k}(\joinVar)] + \frac{\regHyp}{4 \nObs} \boldsymbol{\alpha}^T \mathbf{K} \boldsymbol{\alpha}.
$$

\section{Proof Of Theorem \ref{thm:doubleConsistency}}

The consistency under condition (a) is a direct consequence of the causal learning bound from Theorem \ref{thm:causalBound}. Note that a convergence in high probability as in the bound of Theorem \ref{thm:causalBound} implies the almost sure convergence via Borel-Cantelli's lemma (see the proof of Lemma \ref{lem:firstIpmConv} afterward). So we actually just need to show the convergence toward zero of all the terms in the R.H.S. of the bound in Theorem \ref{thm:causalBound}. First, we show in the proof of Lemma \ref{lem:firstConvDistr} that the \ipm~term will indeed converges to zero. Second, under condition (a) we have that for any $\epsilon >0 $, there exists $\gamma_\epsilon$ such that $\forall \regFct \in \regClass, -2\condExp\regFct + \regFct^2$ is ($\gamma_\epsilon \gamma + \gamma^2$, $\epsilon$)-approximable by $\ipmFctClass$. Hence, we have that for any $\epsilon > 0$: $0 \leq \causalRisk(\regFctMin) - \inf_{\regFct \in \regClass} \causalRisk(\regFct) \leq u_n(\gamma_\epsilon) + \epsilon$, with $u_n(\gamma_\epsilon)$ a series converging to zero for any $\gamma_\epsilon$. As this is true for all $\epsilon > 0$, we have indeed $\causalRisk(\regFctMin) \xrightarrow[n \rightarrow +\infty]{a.s.}  \inf_{\regFct \in \regClass} \causalRisk(\regFct)$.

So actually we mainly need to prove consistency under condition (b). Now that the RKHS $\rkhs$ is of finite dimension we can define its orthogonal basis (for the RKHS inner product) as follows:
\begin{definition} \label{def:othoBasis}
As $\kernel$ is continuous and its RKHS $\rkhs$ is of finite dimension, there exists an orthonormal family of continuous functions $g_{1}, \ldots, g_{K} \in \rkhs$ on $\txSupp$ such that:
\begin{itemize}
    \item If $1 \in \rkhs$ (that is, if the constant function equal to one on $\txSupp$ is in $\rkhs$), then $(1,g_1, \ldots, g_K)$ is an orthogonal family that spans $\rkhs$.
    \item Otherwise, $g_{1}, \ldots, g_{K}$ is simply an orthonormal basis of $\rkhs$.
\end{itemize}
\end{definition}

Note that now we can replace or rewrite the feature embedding as follows $\Phi = \mathbf{g} = (g_k)_{k = 1}^{K}$, and we have accordingly $f = \sum_{k} \lambda_k g_k = \boldsymbol{\lambda}^T \mathbf{g}$ with $\boldsymbol{\lambda} = (\lambda_k)_{k =1}^K$, $\| f \|_{\rkhs} = \| \mathbf{g} \|_2$ and $\inner{f}{\mathbf{a}} = \boldsymbol{\lambda}^T \mathbf{a}$, where now $\mathbf{a} = \int \mathbf{g}(z) d \pi$. We now prove a result that shows that if $\souTarDensity$ is well-specified, that is $\souTarDensity =  \frac{d\rho^*}{d x}(- \mu^0 - \sum_{k=1}^{K} \lambda_k^0 g_k)$ for some $ \mu^0$ and $\boldsymbol{\lambda}^0$, then it turns out that $( \mu^0,\boldsymbol{\lambda}^0)$ is the unique solution of the limit dual problem from Lemma \ref{lem:duality}, with $\nu = \joinDistr$ and $\pi = \txTarget$.
\begin{lemma} \label{lem:converseSol}
Consider a probability density $\phi_0$ such that, for some $(\mu^0, \boldsymbol{\lambda}^0)$, we have $\forall \z \in \txSupp, \phi_0(\z) = \frac{d\rho^*}{d x}(- \mu^0 - \sum_{k=1}^{K} \lambda_k^0 g_k(\z)) $ which takes values in $(0, \clip)$ with non-zero probability under $\nu$. Consider the limit dual problem of Lemma \ref{lem:duality}, that is for $\lambda = 0$ or equivalently when $\nObs \rightarrow + \infty$, and with $\partial\pi / \partial\nu = \phi_0$. Then $(\mu^0, \boldsymbol{\lambda}^0)$ is the unique optimum of the dual problem from Lemma \ref{lem:duality}.
\end{lemma}
\begin{proof}
Recall again that we work on a compact set $\txSupp$ and the $g_k$ functions, being continuous from Definition \ref{def:othoBasis}, are therefore bounded on $\txSupp$. Also from Lemma \ref{lem:legTransRho}, having $\phi_0(\z) \in (0 , \clip)$ is equivalent to $- \mu^0 - \sum_{k=1}^{K} \lambda_k^0 g_k(\z) \in (\frac{d\rho}{d x}(0),\frac{d\rho}{d x}(\clip))$. In particular, there exists $\epsilon>0$, small enough, such that $\nu \left(- \mu^0 - \sum_{k=1}^{K} \lambda_k^0 g_k(\z) \in (\frac{d\rho}{d x}(0) + \epsilon,\frac{d\rho}{d x}(\clip) - \epsilon) \right) > 0$. $\rho^*$ being strictly convex on $(\frac{d\rho}{d x}(0),\frac{d\rho}{d x}(\clip))$, this implies that $\int \rho^*\left(- \mu - \sum_{k=1}^{K} \lambda_k g_k (\z) \right) d \nu(\z)$, as a function of $(\mu, \boldsymbol{\lambda})$, is strictly convex on a neighborhood of $(\mu^0, \boldsymbol{\lambda}^0)$. This would prove the unicity, once we proved that indeed $(\mu^0, \boldsymbol{\lambda}^0)$ is an optimum for the dual problem.

The fact that $(\mu^0, \boldsymbol{\lambda}^0)$ is an optimum is actually straightforward: Just take the differential of the objective of the dual program, and see that it equates 0 (note that in our setting we can interchange integral and differential).
\end{proof}

Recall from Theorem \ref{thm:dualProblem} that the \cbdm~weights can be derived from the solution of the corresponding dual problem. Namely, let $(\mu^n, \boldsymbol{\lambda}^n)$ the solution of the dual problem from Lemma \ref{lem:duality} with $\nu = \hat{\source}^\nObs_{\joinVar}$ and $\mathbf{a} = \hat{\mathbf{a}}(\nObs) \doteq \mathbb{E}_{\empTar}[\mathbf{g}(\joinVar)]$, then $\nObs \w_i(\nObs) = \frac{d\rho^*}{d \z}(- \mu^\nObs - \sum_{k=1}^{K} \lambda_k^\nObs g_k(\joinVar_i))$. Also, the previous lemma shows that actually if $\souTarDensity$ is well-specified, that is $\souTarDensity =  \frac{d\rho^*}{d x}(- \mu^0 - \sum_{k=1}^{K} \lambda_k^0 g_k)$ for some $ \mu^0$ and $\boldsymbol{\lambda}^0$, then it turns out that $( \mu^0,\boldsymbol{\lambda}^0)$ is the unique solution of the limit dual problem from Lemma \ref{lem:duality}, with $\lambda/n = 0$, $\nu = \joinDistr$ and $\mathbf{a} = \mathbb{E}_{\txTarget}[\mathbf{g}(\joinVar)]$. We now provide sufficient conditions such that indeed $(\mu^\nObs, \boldsymbol{\lambda}^\nObs)$ converges toward $(\mu^0, \boldsymbol{\lambda}^0)$.

\begin{lemma} \label{lem:convLambda}
Let $(\mu^\nObs, \boldsymbol{\lambda}^\nObs)$ a solution of the dual problem \ref{lem:duality} with $\nu = \hat{\source}^\nObs_{\joinVar}$ and $\mathbf{a} = \hat{\mathbf{a}}(\nObs)$ (when such optimum doesn't exist just set it to any infinite value, so that the function $1 \wedge \| (\mu^\nObs, \boldsymbol{\lambda}^\nObs) - (\mu^0, \boldsymbol{\lambda}^0) \|_2$ below is actually equal to $1$). Work under Assumption \ref{ass:overlap}, assume also that $\souTarDenBound < \clip$ and $\souTarDensity(\z) = \frac{d\rho^*}{d \z}(- \mu^0 - \sum_{k=1}^{K} \lambda_k^0 g_k(\z))$. Let $\Phi_{\nObs}(\mu, \boldsymbol{\lambda}) = \int \rho^*\left(- \mu - \sum_{k=1}^{K} \lambda_k g_k (\z) \right) d\hat{\source}^\nObs_{\joinVar}(\z)  + \mu + \sum_{i=1}^K \lambda_k \hat{a}_k(\nObs)$, suppose that there exist a neighborhood $\mathcal{B}((\mu^0, \boldsymbol{\lambda}^0),\epsilon)$ of $(\mu^0, \boldsymbol{\lambda}^0)$ and $\beta,C' >0$ independent of $\nObs$ and $\delta$, such that $\forall \nObs >0, \, \delta \in (0,1]$, with probability at least $1-\delta$:
$$
\nObs \geq C' (1 + \log(1/\delta)) \quad \Longrightarrow \quad \Phi_{\nObs} \text{ is $\beta$-strongly convex on } \mathcal{B}((\mu^0, \boldsymbol{\lambda}^0),\epsilon).
$$
Then, there exist $C>0$ independent of $\nObs$ and $\delta$, such that $\forall \nObs>0, \, \delta \in (0,1]$, with probability at least $1-\delta$:
$$
1 \wedge \| (\mu^\nObs, \boldsymbol{\lambda}^\nObs) - (\mu^0, \boldsymbol{\lambda}^0) \|_2 \leq C \sqrt{\frac{1 + \log(1/\delta)}{\nObs}},
$$
where $x\wedge y = \min\{ x,y\}$.
\end{lemma}
\begin{proof}
Assume that $\Phi_{\nObs}$ is $\beta$-strongly convex on $\mathcal{B}((\mu^0, \boldsymbol{\lambda}^0),\epsilon)$ (which is true by assumption with high probability for $\nObs$ large enough). We have for any $(\mu, \boldsymbol{\lambda}) \in \mathcal{B}((\mu^0, \boldsymbol{\lambda}^0),\epsilon)$:
$$
\Phi_{\nObs}(\mu, \boldsymbol{\lambda}) \geq \Phi_{\nObs}(\mu^0, \boldsymbol{\lambda^0}) + \nabla \Phi_{\nObs}(\mu^0, \boldsymbol{\lambda}^0)^\top . ((\mu , \boldsymbol{\lambda}) - ( \mu^0 , \boldsymbol{\lambda}^0 )) + \frac{\beta}{2} \| (\mu , \boldsymbol{\lambda}) - ( \mu^0 , \boldsymbol{\lambda}^0 ) \|^2_2.
$$
In particular, notice that for any $(\mu, \boldsymbol{\lambda}) \in \partial \mathcal{B}((\mu^0, \boldsymbol{\lambda}^0),\epsilon)$ (that is on the boundary):
\begin{equation} \label{eq:stonrgConvBound}
 \Phi_{\nObs}(\mu, \boldsymbol{\lambda}) \geq \Phi_{\nObs}(\mu^0, \boldsymbol{\lambda^0}) - \| \nabla \Phi_{\nObs}(\mu^0, \boldsymbol{\lambda}^0) \|_2 \epsilon  + \frac{\beta}{2} \epsilon^2.  
\end{equation}
Note also that, by assumption, $\nabla \Phi_{\nObs}(\mu^0, \boldsymbol{\lambda}^0) = - \int \mathbf{g}(\z) \souTarDensity(\z) d\hat{\source}^\nObs_{\joinVar}(\z) + \hat{\mathbf{a}}(\nObs)$. By straightforward uses of McDiarmid's inequality (see Lemma \ref{lem:mcDia}) and some union bound, we have the existence of constants $C,C'>0$ such that with probability at least $1-\delta$:
\begin{align} \label{eq:summarizeBound}
    &\| \hat{\mathbf{a}}(\nObs) - \mathbf{a} \|_2  \leq C \sqrt{\frac{1 + \log(1/\delta)}{\nObs}}, \nonumber \\
    & \left \| \int \mathbf{g}(\z) \souTarDensity(\z) d\hat{\source}^\nObs_{\joinVar}(\z) - \mathbf{a} \right\|_2  \leq C' \sqrt{\frac{1 + \log(1/\delta)}{\nObs}}.
\end{align}
Combining \eqref{eq:stonrgConvBound} and \eqref{eq:summarizeBound} yields the fact that there exists a constant $C > 0$ such that with probability at least $1-\delta$, if $\nObs \geq C (1 + \log(1/\delta))$ then we have for any $(\mu, \boldsymbol{\lambda}) \in \partial \mathcal{B}((\mu^0, \boldsymbol{\lambda}^0),\epsilon)$:
$$
\Phi_{\nObs}(\mu, \boldsymbol{\lambda}) + \| \boldsymbol{\lambda}  \|_2 \approxDelta_\nObs > \Phi_{\nObs}(\mu^0, \boldsymbol{\lambda^0}) + \| \boldsymbol{\lambda^0}  \|_2 \approxDelta_\nObs,
$$
which implies not only the existence of $(\mu^\nObs, \boldsymbol{\lambda}^\nObs)$ but that it is actually in $\mathcal{B}((\mu^0, \boldsymbol{\lambda}^0),\epsilon)$, by convexity. For such $(\mu^\nObs, \boldsymbol{\lambda}^\nObs)$ using the strong convexity again, we have:
\begin{equation}\label{eq:otherStrongConv}
    (\nabla \Phi_{\nObs}(\mu^\nObs, \boldsymbol{\lambda}^\nObs) - \nabla \Phi_{\nObs}(\mu^0, \boldsymbol{\lambda}^0) )^\top . ((\mu^\nObs , \boldsymbol{\lambda}^\nObs) - ( \mu^0 , \boldsymbol{\lambda}^0 )) \geq  \beta \| (\mu^\nObs , \boldsymbol{\lambda}^\nObs) - ( \mu^0 , \boldsymbol{\lambda}^0 ) \|^2_2.
\end{equation}
We know that $ \nabla \Phi_{\nObs}(\mu^0, \boldsymbol{\lambda}^0)$ goes to $0$ at speed $O( \sqrt{\frac{1 + \log(1/\delta)}{\nObs}})$. What about $\nabla \Phi_{\nObs}(\mu^\nObs, \boldsymbol{\lambda}^\nObs)$ ? By first order condition we can see that $ - \nabla \Phi_{\nObs}(\mu^\nObs, \boldsymbol{\lambda}^\nObs) = \frac{\lambda}{2n}   \boldsymbol{\lambda}^\nObs $ which converges to $0$ at speed $O( 1/n)$. Hence, using \eqref{eq:otherStrongConv}, we have that there is a constant $C>0$ such that:
$$
\| (\mu^\nObs, \boldsymbol{\lambda}^\nObs) - (\mu^0, \boldsymbol{\lambda}^0) \|_2 \leq C \sqrt{\frac{1 + \log(1/\delta)}{\nObs}}.
$$
\end{proof}
We now show that the conditions from the previous lemma are respected under condition (b) of Theorem \ref{thm:doubleConsistency}.

\subsubsection*{Proof of Theorem \ref{thm:doubleConsistency} under Condition (b)} 

Call $(\mu^\nObs, \boldsymbol{\lambda}^\nObs)$ solutions of the dual problem of Lemma \ref{lem:duality} with $\nu = \hat{\source}^\nObs_{\joinVar}$ and $\mathbf{a} = \hat{\mathbf{a}}(\nObs)$, and let $B>0$ such that $\| (1,\mathbf{g}(\z)) \|_2 \leq B, \forall \z \in \txSupp$, which exists because these functions are continuous and $\txSupp$ is compact. Again by continuity, and the fact that $\frac{d\rho^*}{d x}$ is $L$-Lipschitz (as $\rho^*$ smooth) on finite intervals for some $L>0$, we have for $ (\mu^\nObs, \boldsymbol{\lambda}^\nObs)$ in a neighborhood of $(\mu^0, \boldsymbol{\lambda}^0)$ (using also  Cauchy-Schwarz inequality): 
$$
\left| \nObs \w_i(\nObs) - \souTarDensity(\joinVar_i)\right| = \left| \frac{d\rho^*}{d x}\left(- \mu^\nObs - \sum_{k=1}^{K} \lambda_k^\nObs g_k(\joinVar_i)\right) - \frac{d\rho^*}{d x}\left(- \mu^0 - \sum_{k=1}^{K} \lambda_k^0 g_k(\joinVar_i)\right)\right| \leq B L \| (\mu^\nObs, \boldsymbol{\lambda}^\nObs) - (\mu^0, \boldsymbol{\lambda}^0) \|_2.
$$
This shows that the uniform convergence of the $\nObs \w_i(\nObs)$s toward $\souTarDensity(\joinVar_i)$ would therefore be a direct consequence of Lemma \ref{lem:convLambda}, once we have showed that its assumptions are verified. Actually, we just have to prove that $\Phi_n$ is $\beta$-strongly convex with probability at least $1-\delta$ when $\nObs \geq C (1 + \log(1/\delta))$ for some constant $C>0$. For any $\epsilon>0$, let:
$$
\Psi_\epsilon(\mu, \boldsymbol{\lambda}) = \int \rho^*\left(- \mu - \sum_{k=1}^{K} \lambda_k g_k (\z) \right) \mathbbm{1}_{\z \in A_\epsilon}d\joinDistr(\z), \quad \text{where } A_\epsilon = \{ \z \in \txSupp: - \mu^0 - \sum_{k=1}^{K} \lambda^0_k g_k (\z) \geq \epsilon + \frac{d\rho}{d x}(0)\}.
$$
Because the integral is over a subset of $\txSupp$ where $\frac{d\rho}{d x}(0) < \frac{d\rho}{d x}(0) + \epsilon \leq - \mu^0 - \sum_{k=1}^{K} \lambda^0_k g_k (\z) \leq \souTarDenBound < \clip$, for $(\mu, \boldsymbol{\lambda})$ in a neighborhood of $(\mu^0, \boldsymbol{\lambda}^0)$ we have, by assumption under condition (b), that $\rho^*$ is $M$-strongly convex for some $M>0$ and on some finite interval that includes all values of the functions $-\mu - \boldsymbol{\lambda}^T \mathbf{g}$ and $-\mu^0 - (\boldsymbol{\lambda}^0)^T \mathbf{g}$. Therefore, on this neighborhood:
$$
\nabla^2 \Psi_\epsilon(\mu, \boldsymbol{\lambda}) \succeq M . H_\epsilon , \quad \text{with } H_\epsilon = \left( \int g_k(\z) g_l(\z) \mathbbm{1}_{\z \in A_\epsilon} d \joinDistr(\z) \right)_{k,l \in \{0, \ldots, K \}},
$$
where by convention we set $g_0 = 1$.

We claim that $H_\epsilon$'s lowest eigenvalue is above $8 \beta$ for some $\beta>0$ and $\epsilon$ small enough. Let's prove this fact. Consider the limit of $H_\epsilon$ when $\epsilon \rightarrow 0$, which actually is:
$$
H_0 = \left( \int g_k(\z) g_l(\z) \mathbbm{1}_{\z \in A_0} d \joinDistr(\z) \right)_{k,l \in \{0, \ldots, K \}}, \quad \text{where } A_0 = \{ \z \in \txSupp : \souTarDensity(\z) > 0 \}.
$$
Notice that $A_0$ is dense in $\txSupp$. Indeed, by contradiction assume there exists $\z \in \txSupp$ and $\eta > 0$ s.t.~$\txTarget(\mathcal{B}(\z,\eta))=0$ and $\joinDistr(\mathcal{B}(\z,\eta)) > 0$ (recall $\txSupp$ is the support of $\joinDistr$), which is impossible as by definition $\txTarget = \tDistr \otimes \xDistr$. As $(1,g_1, \ldots, g_K)$ is an independent family of continuous functions on $\txSupp$ there cannot be $(\mu, \boldsymbol{\lambda}) \neq 0$ such that $(\mu, \boldsymbol{\lambda})^\top H_0 (\mu, \boldsymbol{\lambda}) = 0$. This means $H_0$ is positive definite and, as $H_\epsilon \rightarrow_{\epsilon \rightarrow 0} H_0$, so will be $H_\epsilon$ for $\epsilon$ small enough. Hence, $H_\epsilon$'s lowest eigenvalue is above $8 \beta$ for some $\beta>0$ and $\epsilon$ small enough.

Call $\Psi_\epsilon^\nObs(\mu, \boldsymbol{\lambda})$  the empirical counterpart of $\Psi_\epsilon(\mu, \boldsymbol{\lambda})$. For the same reasons we have (on a neighborhood of $(\mu^0, \boldsymbol{\lambda}^0)$): $\nabla^2 \Psi_\epsilon^\nObs(\mu, \boldsymbol{\lambda}) \succeq M.  H^\nObs_\epsilon $, where $H^\nObs_\epsilon = \left( \int g_k(\z) g_l(\z) \mathbbm{1}_{\z \in A_\epsilon} d\hat{\source}^\nObs_{\joinVar}(\z) \right)_{k,l \in \{0, \ldots, K \}}$. By McDiarmid's Inequality and using an union bound we can see that with high probability $\| H_\epsilon^\nObs - H_\epsilon\|_2 \leq C \sqrt{\frac{1 + \log(1/\delta)}{\nObs}}$ and in particular $\Psi_\epsilon^\nObs$ will be $\beta$-strongly convex for $\nObs \geq C' (1 + \log(1/\delta))$ with probability at least $1-\delta$. As, on a fixed neighborhood of $(\mu^0, \boldsymbol{\lambda}^0)$ we have:
$$
\Psi_\epsilon^\nObs(\mu, \boldsymbol{\lambda}) = \int \rho^*\left(- \mu - \sum_{k=1}^{K} \lambda_k g_k (\z) \right) \mathbbm{1}_{\z \in A_\epsilon} d\hat{\source}^\nObs_{\joinVar}(\z),
$$
the conditions of Lemma \ref{lem:convLambda} are verified and from our previous steps, we have that there exists a constant $C>0$ s.t.~$\forall \nObs > 0, \, \delta \in (0,1]$, with probability at least $1-\delta$:
\[
\forall i, \quad \left| \nObs \w_i(\nObs) - \souTarDensity(\joinVar_i) \right| \leq C \sqrt{\frac{1 + \log(1/\delta)}{\nObs}}.
\]

Actually the convergence under $(b)$ could be derived by following some very common steps that can be found in leaning theory books [see \cite{shalev2014understanding}, Chapter 26], therefore we should only briefly mention, and roughly, what are these steps. First notice that, because $\regClass$ is bounded and $\ySupp$ compact,  we have with probability least $1-\delta$:
\begin{equation} \label{eq:thm5bound}
  \exists C>0, \quad \forall g \in \regClass, \quad \left|\sum_{i=1}^{\nObs} \w_i (\out_i - \regFct(\treat_i))^2 - \frac{1}{\nObs} \sum_{i=1}^{\nObs} \souTarDensity(\joinVar_i) (\out_i - \regFct(\treat_i))^2 \right| \leq C \sqrt{\frac{1+\log(1/\delta)}{\nObs}}.  
\end{equation}
Therefore we might just consider the empirical risk weighted by $\souTarDensity$ instead, which can be rewritten as an empirical risk under another loss function. More precisely:
$$
\frac{1}{\nObs} \sum_{i=1}^{\nObs} \souTarDensity(\joinVar_i) (\out_i - \regFct(\treat_i))^2 = \frac{1}{\nObs} \sum_{i=1}^{\nObs} l(g, \joinVar_i, \out_i),
$$
where $l(g, \joinVar_i, \out_i) =  \souTarDensity(\joinVar_i) (\out_i - \regFct(\treat_i))^2$ is a bounded loss function (recall also Assumption \ref{ass:overlap}). By Theorem 26.5 from \cite{shalev2014understanding} (and also contraction lemma, Lemma 26.9 from the same reference), we have a bound of the following form, with probability at least $1-\delta$, for all $\regFct \in \regClass$:
$$
\mathbb{E}_{\source}[l(g,\joinVar,\out)] - \frac{1}{\nObs} \sum_{i=1}^{\nObs} l(g, \joinVar_i, \out_i) \leq C' \radComp + C'' \sqrt{\frac{1+\log(1/\delta)}{\nObs}},
$$
where $C', C''>0$ are two constants. Notice that actually that $\mathbb{E}_{\source}[l(g,\joinVar,\out)] = \mathbb{E}_{\target}[(\out - \regFct(\treat))^2] = \causalRisk(g) + c$ ($c$ being a constant independent of $g$) ! Combining with \eqref{eq:thm5bound} yields the convergence of $\causalRisk(\regFctMin)$ to $\inf_{\regFct \in \regClass} \causalRisk(\regFct)$ (see equation (26.10) from \cite{shalev2014understanding}).

\section{Proof Of Theorem \ref{thm:distrConv}}
In the following we will denote $\joinVar = (\treat, \cov)$. First let's recall McDiarmid's inequality (e.g.~\cite{shalev2014understanding}, Lemma 26.4), generalization of Hoeffding's inequality:
\begin{lemma}[McDiarmid's Inequality] \label{lem:mcDia}
Let $\joinVar_{1}, \ldots, \joinVar_{\nObs}$ independent random variables, $c_{i}>0$ for $i \in \{1, \dots, \nObs \}$ and let $\fct$ a function such that for any $\z_{i}, \z'_{i} \in \txSupp$ with $i$ in $\{1, \dots, \nObs \}$ we have for all $i$
$
\quad \left| \fct(\z_1, \ldots, \z_\nObs) - \fct(\z_1, \ldots, \z_{i-1}, \z'_{i},\z_{i+1}, \ldots, \z_\nObs) \right| \leq c_i.
$
Then,$$
\mathbb{P}\left( \left| f(\joinVar_{1}, \ldots, \joinVar_{\nObs}) - \mathbb{E}[f(\joinVar_{1}, \ldots, \joinVar_{\nObs})]\right| \geq \epsilon \right) \leq 2 \exp\left( \frac{-2 \epsilon^2}{\sum_{i=1}^{\nObs} c_{i}^2} \right).
$$
\end{lemma}
We will need three lemmas to prove Theorem \ref{thm:distrConv}. The first states that if we knew the density ratio $\souTarDensity$ and use it as importance weights the resulting re-weighted empirical distribution will indeed converges to $\txTarget$ w.r.t.~the considered \ipm. And the same goes for $\empTar^\nObs$.
\begin{lemma} \label{lem:firstIpmConv}
Let $\ipm \in \{ \wass, \mmd_\kernel\}$ where $\kernel$ is a bounded kernel on $\txSupp \subset \mathbb{R}^{\joindim}$ compact. Consider $\txTarget = \tDistr \otimes \xDistr$ and $\joinVar_1, \joinVar_2, \dots$ an i.i.d.~sequence from $\joinDistr$. Assume Assumption \ref{ass:overlap} and let $\phi = \souTarDensity \leq \souTarDenBound$. Define $\hat{\source}^{\nObs}_{\phi} = \sum_{i = 1}^{\nObs} \phi(\joinVar_i) \dirac{\joinVar_i} / \sum_{i = 1}^{\nObs} \phi(\joinVar_i)$ and $\empTar^{\nObs}$ based on the $\nObs$ first samples $(\joinVar_{i})_{i = 1}^{\nObs}$. Then, for $\hat{\nu}_{\nObs} \in \{ \hat{\source}^{\nObs}_{\phi}, \empTar^{\nObs}\}$, we have:
$$
\ipm(\hat{\nu}_{\nObs}, \txTarget) \xrightarrow[\nObs \rightarrow + \infty]{a.s.} 0.
$$
\end{lemma}
\begin{proof}
Let's start with $\ipm = \wass$. $\txSupp$ is compact, therefore convergence w.r.t.~$\wass$ is equivalent to convergence in distribution (see \cite{villani2003topics}, Theorem 7.12). Let $g$ continuous on $\txSupp$, thus also $\| g \|_{\infty} \leq B/2$ for some $B \geq 0$. Applying (twice) the law of large numbers it is easy to see that:
\begin{equation}\label{eq:LLN}
\mathbb{E}_{\hat{\source}^{\nObs}_{\phi}}[g(\joinVar)] = \sum_{i = 1}^{\nObs} \phi(\joinVar_i) g(\joinVar_i) / \sum_{i = 1}^{\nObs} \phi(\joinVar_i) \xrightarrow[\nObs \rightarrow + \infty]{a.s.} \mathbb{E}_{\txTarget}[g(\joinVar)].
\end{equation}
Also, let $\fct(\joinVar_1, \ldots, \joinVar_\nObs) = \mathbb{E}_{\empTar^\nObs}[g(\joinVar)] = \frac{1}{\nObs^2} \sum_{i,j = 1}^\nObs g(\treat_{i}, \cov_{j})$. We are going to apply McDiarmid's inequality to $\fct$. As in Lemma \ref{lem:mcDia}, let $(\z_i)_{i = 1}^{\nObs} = (\tr_{i}, \x_{i})_{i = 1}^{\nObs}$ and $\z'_{k} = (\tr'_{k}, \x'_{k})$ for some $k$:
\begin{align*}
    \left| \fct(\z_1, \ldots, \z_{\nObs}) - \fct(\z_1, \ldots, \z'_k, \ldots, \z_{\nObs})\right| = & \left| \frac{1}{\nObs^2} \sum_{i,j \neq k} g(\tr_{i}, \x_{j}) + \frac{1}{\nObs^2} \sum_{i \neq k}^{\nObs} g(\tr_{i}, \x_{k})+ \frac{1}{\nObs^2} \sum_{j \neq k} g(\tr_{k}, \x_{j}) + \frac{g(\tr_{k}, \x_{k})}{\nObs^2} \right. \\
    & \left. - \frac{1}{\nObs^2} \sum_{i,j \neq k} g(\tr_{i}, \x_{j}) - \frac{1}{\nObs^2} \sum_{i \neq k}^{\nObs} g(\tr_{i}, \x'_{k})- \frac{1}{\nObs^2} \sum_{j \neq k} g(\tr'_{k}, \x_{j}) - \frac{g(\tr'_{k}, \x'_{k})}{\nObs^2} \right| \\
    \leq & B (2 \nObs - 1)/\nObs^2.
\end{align*}
Hence, $\fct$ satisfies the condition of McDiarmid's inequality with $c_i = B (2 \nObs - 1)/\nObs^2$. We thus have:
$$
\forall \epsilon > 0, \quad \mathbb{P}\left( \left| f(\joinVar_{1}, \ldots, \joinVar_{\nObs}) - \mathbb{E}[f(\joinVar_{1}, \ldots, \joinVar_{\nObs})]\right| \geq \epsilon \right) \leq 2 \exp\left( -\frac{2 \epsilon^2}{B^2} \left( \frac{2\nObs - 1}{\nObs} \right)^2 \nObs \right).
$$
The RHS being summable w.r.t.~$\nObs$, Borel-Cantelli's lemma (e.g. Proposition 2.6 from \cite{ccinlar2011probability}) (note we will use this lemma several times) implies the almost sure convergence of $\left| f(\joinVar_{1}, \ldots, \joinVar_{\nObs}) - \mathbb{E}[f(\joinVar_{1}, \ldots, \joinVar_{\nObs})]\right|$ toward 0. Furthermore,
$$
\mathbb{E}[f(\joinVar_{1}, \ldots, \joinVar_{\nObs})] = \frac{\nObs(\nObs - 1)}{\nObs^2} \mathbb{E}_{\txTarget}[g(\joinVar)] + \frac{1}{\nObs}\mathbb{E}_{\joinDistr}[g(\joinVar)] \xrightarrow[\nObs \rightarrow + \infty]{}  \mathbb{E}_{\txTarget}[g(\joinVar)].
$$
Hence: 
\begin{equation} \label{eq:mcdConv}
    \mathbb{E}_{\empTar^\nObs}[g(\joinVar)] \xrightarrow[\nObs \rightarrow + \infty]{a.s.} \mathbb{E}_{\txTarget}[g(\joinVar)].
\end{equation}
Combining \eqref{eq:LLN} and \eqref{eq:mcdConv}, as $g$ was chosen arbitrarily, for $\hat{\nu}_{\nObs} \in \{ \hat{\source}^{\nObs}_{\phi}, \empTar^{\nObs}\}$ we have that almost surely $\hat{\nu}_{\nObs}$ will converge in distribution to $\txTarget$ (see Remark \ref{rem:countability}). Therefore, $\wass(\hat{\nu}_{\nObs}, \txTarget) \xrightarrow[\nObs \rightarrow + \infty]{a.s.} 0.$

Now let's turn to $\ipm = \mmd_{\kernel}$. Let $B>0$ such that $\kernel \leq B/2$ on $\txSupp$, as we assumed it's bounded. It is well-known that $\forall \mu, \nu$ probability measures we have (e.g.~see \cite{gretton2007kernel, gretton2012kernel,ramdas2017wasserstein}):
\begin{equation}\label{eq:MMDextension}
    \mmd^2_{\kernel}(\mu, \nu) = \mathbb{E}_{\mu \otimes \mu}[\kernel(\joinVar,\joinVar')] - 2 \mathbb{E}_{\mu \otimes \nu}[\kernel(\joinVar,\joinVar')] + \mathbb{E}_{\nu \otimes \nu}[\kernel(\joinVar,\joinVar')],
\end{equation}
and in particular :
$$
\mmd^2_{\kernel}(\empTar^{\nObs}, \txTarget) = \mathbb{E}_{\empTar^{\nObs} \otimes \empTar^{\nObs}}[\kernel(\joinVar,\joinVar')] - 2 \mathbb{E}_{\empTar^{\nObs} \otimes \txTarget}[\kernel(\joinVar,\joinVar')] + \mathbb{E}_{\txTarget \otimes \txTarget}[\kernel(\joinVar,\joinVar')].
$$
We only need to study the convergence of the first two terms, let's focus on the first term only as the proof for the other is essentially the same. We will also omit the proof for $\mmd_{\kernel}(\hat{\source}^{\nObs}_{\phi}, \txTarget) \xrightarrow[]{a.s.} 0$ as, again, this is also a mere application of McDiarmid's inequality (recall $\phi \leq \souTarDenBound$). Over all the terms we need to study the convergence, $\mathbb{E}_{\empTar^{\nObs} \otimes \empTar^{\nObs}}[\kernel(\joinVar,\joinVar')]$ is probably the most difficult, so if we had to provide the proof for at least one of them, it's normal we should look at this one in priority. 

Let $\fct(\joinVar_1, \ldots, \joinVar_\nObs) = \mathbb{E}_{\empTar^{\nObs} \otimes \empTar^{\nObs}}[\kernel(\joinVar,\joinVar')] = \frac{1}{\nObs^2} \sum_{i,j=1}^{\nObs} \frac{1}{\nObs^2} \sum_{k,l=1}^{\nObs} \kernel((\treat_i,\cov_j),(\treat_k,\cov_l))$. As in Lemma \ref{lem:mcDia}, let $(\z_i)_{i = 1}^{\nObs} = (\tr_{i}, \x_{i})_{i = 1}^{\nObs}$ and $\z'_{m} = (\tr'_{m}, \x'_{m})$ for some $m$:
$$
\left|  \fct(\z_1, \ldots, \z_{\nObs}) - \fct(\z_1, \ldots, \z'_m, \ldots, \z_{\nObs}) \right| \leq \frac{1}{\nObs^2} \sum_{i,j \, s.t. \, i=m \, or \, j=m} \frac{1}{\nObs^2} \sum_{k,l \, s.t. \, k=m \, or \, l=m} B \leq \left( \frac{2\nObs - 1}{\nObs^2}\right)^2 B.
$$
Also, 
$$\mathbb{E}[\fct(\joinVar_1, \ldots, \joinVar_\nObs)] = \frac{1}{\nObs^2} \sum_{i \neq j}^\nObs \frac{1}{\nObs^2} \sum_{i \neq j, \, k,l \notin \{ i,j\}} \mathbb{E}[\kernel((\treat_i, \cov_j),(\treat_k, \cov_l))] + O(1/\nObs) \sim_\nObs \mathbb{E}_{\txTarget \otimes \txTarget}[\kernel(\joinVar,\joinVar')].$$

Applying McDiarmid's inequality and Borel-Cantelli's lemma we have that: $$\mathbb{E}_{\empTar^{\nObs} \otimes \empTar^{\nObs}}[\kernel(\joinVar,\joinVar')] \xrightarrow[\nObs \rightarrow + \infty]{a.s.} \mathbb{E}_{\txTarget \otimes \txTarget}[\kernel(\joinVar,\joinVar')].$$
Also, as mentioned before we can prove the same convergence for $\mathbb{E}_{\empTar^{\nObs} \otimes \txTarget}[\kernel(\joinVar,\joinVar')]$. This concludes our proof.
\end{proof}

\begin{remark}\label{rem:countability}
We have claimed that a series of probability measures $\hat{\nu}_{\nObs}$ built on $(\joinVar_{i})_{i = 1}^{\nObs}$ almost surely (w.r.t.~the sequence $(\joinVar_{i})_{i \geq 1}$) converges in distribution to $\txTarget$ by showing that for any fixed continuous function $g$ then, almost surely, $\mathbb{E}_{\hat{\nu}_{\nObs}}[g(\joinVar)] \xrightarrow[\nObs \rightarrow + \infty]{a.s.} \mathbb{E}_{\txTarget}[g(\joinVar)]$. This may look like it's not enough as the almost sure convergence in distribution means that, almost surely, for \emph{all} continuous function $g$ we have the convergence of the expectation, especially because the space of all continuous functions on $\txSupp$ is in general not countable. As a matter of fact, the two statements are equivalent in our setting, and this is because $C(\txSupp, \|.\|_{\infty})$, $\txSupp$ being compact, is separable [\cite{dudley2018real}]. Indeed, it means that $C(\txSupp, \|.\|_{\infty})$ admits a countable basis that can approximate any of its functions, and it is easy to see that our proof actually implies that almost surely the convergence in expectation will hold for all of these basis functions, which is enough for the convergence in distribution. We will use this fact (implicitly) again in the following proofs.
\end{remark}

Because, almost surely, the weights based on $\souTarDensity$ become feasible for $\nObs$ large enough, Lemma \ref{lem:firstIpmConv} implies that the objective function from the \cbdm~minimization program of Definition \ref{def:cbdmIPM} converges toward zero almost surely.

\begin{lemma} \label{lem:optimObjConv}
Consider the same notations and assumptions as from Lemma \ref{lem:firstIpmConv}. Assume furthermore that $\clip > \souTarDenBound$. Then for $\ipm \in \{ \wass, \mmd_{\kernel} \}$:
$$
\min_{\wVect \in \simplex; \,\, \w_i \leq \clip / \nObs, \, \forall i} \Phi_{\nObs} (\wVect) = \ipm^2(\wDistr, \empTar^\nObs) + \regHyp \| \wVect \|_2^2  \xrightarrow[\nObs \rightarrow + \infty]{a.s.} 0. 
$$
\end{lemma}
\begin{proof}
We just have to prove that almost surely there are, for $\nObs$ large enough, some feasible vectors $\wVect(\nObs)$ such that $\Phi_{\nObs} (\wVect(\nObs)) \xrightarrow[\nObs \rightarrow + \infty]{a.s.} 0 $. Based on the previous lemma, consider $\wVect(\nObs)$ defined as $\w_i(\nObs) = \phi(\joinVar_i) / \sum_{i = 1}^{\nObs} \phi(\joinVar_i), \forall i$.

Note that any \ipm~respects the triangular inequality, therefore we have:
$$
\ipm^2(\hat{\source}^{\nObs}_{\phi}, \empTar^\nObs) \leq 2. \ipm^2(\hat{\source}^{\nObs}_{\phi}, \txTarget) + 2. \ipm^2(\empTar^{\nObs}, \txTarget) \xrightarrow[\nObs \rightarrow + \infty]{a.s.} 0 ,
$$
where the almost sure convergence toward zero came from Lemma \ref{lem:firstIpmConv}. Also, by law of large numbers:
$$
\| \wVect(\nObs) \|_2^2 = \frac{1}{\nObs^2} \sum_{i = 1}^{\nObs}\phi^2(\joinVar_i) \Big/ \left(\frac{1}{\nObs} \sum_{i = 1}^{\nObs} \phi(\joinVar_i) \right)^2 \leq \frac{1}{\nObs} \souTarDenBound^2 \Big/ \left(\frac{1}{\nObs} \sum_{i = 1}^{\nObs} \phi(\joinVar_i)\right)^2 \xrightarrow[\nObs \rightarrow + \infty]{a.s.} 0.
$$
We just have to prove now that almost surely for $\nObs$ large enough $\wVect(\nObs)$ will be feasible, that is $\w_i(\nObs) \leq \clip / \nObs, \forall i$. At least we have that $\w_i(\nObs) \leq \souTarDenBound / \sum_{i = 1}^{\nObs} \phi(\joinVar_i), \forall i$. Furthermore, using (one-sided) McDiarmid's inequality:
\begin{equation}\label{eq:feasibility}
\mathbb{P}\left(\souTarDenBound / \sum_{i = 1}^{\nObs} \phi(\joinVar_i) > \clip \Big/ \nObs \right) = \mathbb{P}\left( 0 > \frac{\souTarDenBound - \clip}{\clip}  > \frac{1}{\nObs} \sum_{i = 1}^{\nObs} \phi(\joinVar_i) - 1 \right) \leq \exp\left(-\frac{(\souTarDenBound - \clip)^2}{2 \souTarDenBound^2 \clip^2} \nObs \right).
\end{equation}
Applying Borel-Cantelli's lemma, we conclude our proof. 
\end{proof}
Now that we have proved the almost sure convergence to zero of the objective function (from Definition \ref{def:cbdmIPM}), this implies that the difference between the \cbdm~weighted empirical distribution and the target distribution goes also to zero, w.r.t.~the considered \ipm s. Hence, when $\ipm \in \{ \wass, \mmd_\kernel\}$ where $\kernel$ is a universal kernel, this implies almost surely the convergence in distribution (w.r.t.~$\treat$ and $\cov$ for now).
\begin{lemma} \label{lem:firstConvDistr}
Consider again the notations and assumptions from Lemma \ref{lem:firstIpmConv} and Lemma \ref{lem:optimObjConv}. Assume also that the kernel $\kernel$ is universal. Let $\wVect(\nObs)$ be  the \cbdm~weights from Definition \ref{def:cbdmIPM}, that is $\wVect(\nObs) =  \argmin_{\wVect \in \simplex; \,\, \w_i \leq \clip / \nObs, \, \forall i} \Phi_{\nObs} (\wVect)$ for $\ipm \in \{ \wass, \mmd_{\kernel} \}$. Then:
$$
a.s. \quad \hat{\source}^{\wVect(\nObs)}_{\treat, \cov} \xrightarrow[n \rightarrow +\infty]{d} \txTarget.
$$
\end{lemma}
\begin{proof}
As the kernel is universal, hence continuous, it is bounded on $\txSupp$ compact. From Lemmas \ref{lem:firstIpmConv} and \ref{lem:optimObjConv} we get:
$$
\ipm(\hat{\source}^{\wVect(\nObs)}_{\treat, \cov}, \txTarget) \leq \ipm(\empTar^\nObs, \txTarget) + \ipm(\hat{\source}^{\wVect(\nObs)}_{\treat, \cov},\empTar^\nObs) \xrightarrow[\nObs \rightarrow + \infty]{a.s.} 0.
$$
Hence $\ipm(\hat{\source}^{\wVect(\nObs)}_{\treat, \cov}, \txTarget) \xrightarrow[\nObs \rightarrow + \infty]{a.s.} 0$ and especially for $\ipm = \wass$ this implies directly the convergence in distribution (Theorem 7.12 from \cite{villani2003topics}).

Now let's look at $\ipm = \mmd_\kernel$. Let $g$ continuous and $\epsilon > 0$. Because $\kernel$ is universal, calling $\rkhs$ its related RKHS, we have that:
\begin{equation}\label{eq:approxFctRKHS}
    \exists h \in \rkhs \text{ s.t. } \| g - h\|_{\infty} \leq \epsilon /3.
\end{equation}
Furthermore, as $\mmd_{\kernel}(\hat{\source}^{\wVect(\nObs)}_{\treat, \cov}, \txTarget) = \sup_{h \in \rkhs \setminus \{ 0\} } | \mathbb{E}_{\hat{\source}^{\wVect(\nObs)}_{\treat, \cov}}[h(\joinVar)] - \mathbb{E}_{\txTarget}[h(\joinVar)]  | / \| h \|_{\rkhs} \xrightarrow[\nObs \rightarrow + \infty]{a.s.} 0$, it means that almost surely for $\nObs$ big enough:
\begin{equation}\label{eq:smallMMD}
    \left|\mathbb{E}_{\hat{\source}^{\wVect(\nObs)}_{\treat, \cov}}[h(\joinVar)] - \mathbb{E}_{\txTarget}[h(\joinVar)] \right | \leq \epsilon / 3.
\end{equation}
Combining \eqref{eq:approxFctRKHS} and \eqref{eq:smallMMD}, it is easy to see that $|\mathbb{E}_{\hat{\source}^{\wVect(\nObs)}_{\treat, \cov}}[g(\joinVar)] - \mathbb{E}_{\txTarget}[g(\joinVar)] | \leq \epsilon$. As $g$ and $\epsilon$ are chosen arbitrarily this implies the convergence in distribution.
\end{proof}

We have the convergence in law of the weighted empirical distribution w.r.t.~$\treat$ and $\cov$. Because the \cbdm~weights are honest (they don't depend on outcome values from the observational data), as we show below it means that actually the full weighted empirical distribution (w.r.t.~$\treat$, $\cov$ and $\out$) will converge to $\target$. This concludes our proof of Theorem \ref{thm:distrConv}.

\subsubsection*{Proof of Theorem \ref{thm:distrConv}} We proved from Lemma \ref{lem:firstConvDistr} that a.s. $\hat{\source}^{\wVect(\nObs)}_{\treat, \cov} \xrightarrow[n \rightarrow +\infty]{d} \txTarget$, if $\wVect(\nObs)$ is the solution from the \cbdm~minimization defined in Definition \ref{def:cbdmIPM}. Let $\hat{\source}^{\wVect(\nObs)}_{} = \sum_{i = 1}^{\nObs} \w_{i}(\nObs) \dirac{(\joinVar_i,\out_i)}$, we want to show that (a.s.) $\hat{\source}^{\wVect(\nObs)} \xrightarrow[n \rightarrow +\infty]{d} \target $.

Let $g$ continuous on $\txSupp \times \ySupp$, which is compact, and let $B$ s.t.~$\| g \|_{\infty} \leq B/2$. Note that $\mathbb{E}_{\hat{\source}^{\wVect(\nObs)}}[g(\joinVar, \out)] = \sum_{i = 1}^{\nObs} \w_{i}(\nObs) g(\joinVar_i, \out_i)$. Recall that $\forall i, \w_{i}(\nObs) \leq \clip / \nObs$ and they depend only on $(\joinVar_{i})_{i=1}^{\nObs}$. Therefore, applying McDiarmid's inequality:
$$
\mathbb{P}\left( \left| \sum_{i = 1}^{\nObs} \w_{i}(\nObs) g(\joinVar_i, \out_i) - \sum_{i = 1}^{\nObs} \w_{i}(\nObs) \mathbb{E}[g(\joinVar_i, \out_i)|\joinVar_i] \right| \geq \epsilon \Big|(\joinVar_{i})_{i=1}^{\nObs} \right) \leq 2 \exp \left( - \frac{2 \epsilon^2}{B^2 \clip^2}\nObs\right).
$$
By Borel-Cantelli's lemma we have (a.s.) for $\nObs$ large enough: $\left| \sum_{i = 1}^{\nObs} \w_{i}(\nObs) (g(\joinVar_i, \out_i) - \mathbb{E}[g(\joinVar_i, \out_i)|\joinVar_i]) \right| \leq \epsilon$. $\mathbb{E}[g(\joinVar, \out)|\joinVar])$ might not be continuous w.r.t.~$\joinVar$ though, hence we can't apply the convergence in distribution from Lemma \ref{lem:firstConvDistr} directly to it. Instead, Lusin theorem states that there exists $g^*$ continuous on $\txSupp$ compact s.t.~$g^{*} \neq \mathbb{E}[g(\joinVar, \out)|\joinVar])$ with probability at most $\epsilon / (2 B \clip)$ (w.r.t.~$\joinDistr$) and $\| g^{*} \|_{\infty} \leq B/2$. We can see that:
\begin{align*}
    \left|\sum_{i = 1}^{\nObs} \w_{i}(\nObs) g(\joinVar_i, \out_i) - \mathbb{E}_{\target}[g(\joinVar, \out)]\right| \leq & \left| \sum_{i = 1}^{\nObs} \w_{i}(\nObs) (g(\joinVar_i, \out_i) - \mathbb{E}[g(\joinVar_i, \out_i)|\joinVar_i]) \right| +  \left| \sum_{i = 1}^{\nObs} \w_{i}(\nObs) (g^*(\joinVar_i) - \mathbb{E}[g(\joinVar_i, \out_i)|\joinVar_i]) \right| \\
    & + \left| \sum_{i = 1}^{\nObs} \w_{i}(\nObs) g^*(\joinVar_i) - \mathbb{E}_{\txTarget}[g^*(\joinVar)] \right| +  \left|\mathbb{E}_{\txTarget}\left[g^*(\joinVar) - \mathbb{E}[g(\joinVar, \out)|\joinVar]\right] \right| \\
    \leq & \epsilon +  \frac{B\clip}{\nObs} \sum_{i=1}^{\nObs} \w_i(\nObs) \mathbbm{1}_{g^*(\joinVar_i) \neq \mathbb{E}[g(\joinVar_i, \out_i)|\joinVar_i]} + \left| \sum_{i = 1}^{\nObs} \w_{i}(\nObs) g^*(\joinVar_i) - \mathbb{E}_{\txTarget}[g^*(\joinVar)] \right| + \frac{\epsilon}{2}.
\end{align*}
The second term, by law of large number, will be bellow $\epsilon$ (a.s.) for $\nObs$ large enough. We can say the same for the third term by Lemma \ref{lem:firstConvDistr}. As $\epsilon$ and $g$ are chosen arbitrarily this concludes our proof.

\section{Algorithm For $\ipm = \wass$} \label{sec:appendixAlgo}

We give here the explicit formulation of the optimization problem from Definition \ref{def:cbdmIPM} when $\ipm = \wass$. As said before, this is a quadratic program, at least positive semi-definite when $\regHyp = 0$, and positive definite when $\regHyp > 0$. This is a consequence from the fact that optimal transport (which gives rise to the Wasserstein distance) is essentially a linear program (\cite{villani2003topics}, see also \cite{peyre2019computational} section 2.3). More precisely, let $\empTar$ be a weighted sum of Diracs $\sum_{j = 1}^{m} q_{j} \dirac{\z_j}$. Hence, the optimization program here is:
$$
\min_{ \substack{ \wVect, M \in \mathcal{M}_{+}(n,m) \\ M \mathbf{1} = \wVect, \, M^\top \mathbf{1} = \mathbf{q}} } \left( \sum_{i,j} M_{i,j} \| \joinVar_i - \z_j \|_2 \right)^2 + \regHyp \| \wVect \|_2^2,
$$
where $\mathbf{q} = (q_j)_{j=1}^m$ and we called $ \mathcal{M}_{+}(n,m)$ the set of non-negative matrices of dimension $n \times m$.

\end{document}